\documentclass{birkau}
\usepackage[T1]{fontenc}
\usepackage[utf8]{inputenc}
\usepackage{amsmath,amssymb,url,graphicx}
\usepackage{enumitem} 
\newcommand*\sigmac{\includegraphics[width=.8em]{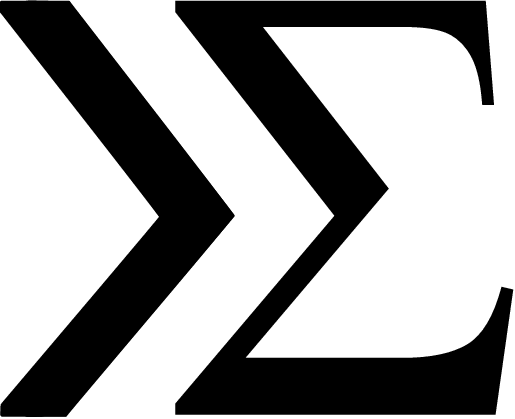}}
\numberwithin{equation}{section}
\theoremstyle{plain}
\newtheorem{theorem}{Theorem}[section]
\newtheorem{lemma}[theorem]{Lemma}
\newtheorem{proposition}[theorem]{Proposition}
\newtheorem{corollary}[theorem]{Corollary}

\theoremstyle{definition}
\newtheorem{definition}[theorem]{Definition}

\newtheorem{remark}[theorem]{Remark}
 
\newtheorem{assumption}[theorem]{Assumption}
\newtheorem{example}[theorem]{Example}

\usepackage{hyperref}
\usepackage[nameinlink,noabbrev]{cleveref}

\crefname{theorem}{}{}
\Crefname{theorem}{}{}
\crefformat{theorem}{#2#1#3}    

\crefformat{equation}{(#2#1#3)}

\begin{document}
\title[The Multisign Algebra: A Generalization of the Sign Concept]{The Multisign Algebra: A Generalization of the Sign Concept}
\author[S. Aliaga-Rojas]{Sebastián Aliaga-Rojas}
\address{Departamento de Ingeniería Informática\\Facultad de Ingeniería\\ Universidad de Santiago de Chile\\Av. Alameda Libertador Bernardo O'Higgins 3363\\Estación Central\\Santiago\\Región Metropolitana\\Chile}
\email{sebastian.aliaga.r@usach.cl}
\author[P. Landero-Sepúlveda]{Pamela Landero-Sepúlveda}
\address{Departamento de Ingeniería Informática\\Facultad de Ingeniería\\ Universidad de Santiago de Chile\\Av. Alameda Libertador Bernardo O'Higgins 3363\\Estación Central\\Santiago\\Región Metropolitana\\Chile}
\email{pamela.landero@usach.cl}
\author[M. Inostroza-Ponta]{Mario Inostroza-Ponta}
\address{Departamento de Ingeniería Informática\\Facultad de Ingeniería\\ Universidad de Santiago de Chile\\Av. Alameda Libertador Bernardo O'Higgins 3363\\Estación Central\\Santiago\\Región Metropolitana\\Chile}
\email{mario.inostroza@usach.cl}
\begin{abstract}
We introduce Multisign Algebra, a generalization of the
classical sign concept in which each nonzero element
combines a magnitude with one of $s\geq2$ discrete signs.
We extend a single-valued additive construction
on metric spiders to abstract magnitude monoids
and equip the resulting numerical system with
an internal multiplication.
Under addition, elements with distinct signs and equal
magnitudes sum to zero. Addition is commutative and
satisfies signed-associativity, while each nonzero
element admits exactly $s-1$ bilateral additive inverses.
We prove that this number attains the upper bound
imposed by the axioms of an $n$i-signed-group on the
same canonically labelled multisign set with zero as its identity. We characterize the
$\binom{s}{2}$ maximal associative additive
subsemigroups as abelian groups isomorphic to the
Grothendieck group of the magnitude monoid and
establish that equations of the form $x\oplus u=v$
admit either one or $s-1$ solutions. Multiplication is associative,
commutative, and distributive over addition, and the
group of nonzero elements under multiplication is isomorphic to
$\mathbb{Z}/s\mathbb{Z}\times P^{\times}$, where
$P^{\times}$ is the multiplicative group of nonzero magnitudes.
We characterize multiplicative orders and roots, and
show that multiplication by nonzero elements induces
automorphisms of the additive structure. We define generalized rings, fields, modules,
vector spaces, and algebras with a prescribed number of bilateral additive
inverses per nonzero element.
We prove that the full magnitude assumptions characterize
the nonnegative cones of ordered fields and that
the two-sign construction over such a cone recovers
the corresponding ordered field. Finally, we establish
that the multisign construction contains a classical
subfield sharing its multiplicative identity if and
only if $s$ is even.
\end{abstract}
\maketitle

\section{Introduction}\label{introduction}

In classical algebra, associativity and the existence of
an identity guarantee the uniqueness of bilateral
inverses whenever they exist. Relaxing associativity
allows us to consider algebraic structures in which
elements admit multiple additive inverses.

Each nonzero real number combines a magnitude with one
of two signs, according to whether it is positive or negative.
We extend this representation by treating signs as
discrete labels independent of the order of magnitudes.

We define a multisign number as ${^{d}p}$, where $d$
denotes its sign and $p$ its magnitude. The number of
available nonzero signs is $s\geq 2$, and ${^{0}0}$
denotes the additive identity.

Addition combines magnitudes when the signs coincide
and subtracts magnitudes when the signs differ,
retaining the sign of the larger magnitude. Elements
with different signs and equal magnitudes cancel.

For example, consider a system with three nonzero
signs and magnitudes in $\mathbb{R}_{\geq 0}$:
\begin{align*}
    {^{1}2}\oplus{^{1}3} &= {^{1}5},\\
    {^{1}3}\oplus{^{2}2} &= {^{1}1},\\
    {^{1}1}\oplus{^{2}1} &= {^{0}0},\\
    {^{1}1}\oplus{^{3}1} &= {^{0}0}.
\end{align*}

The last two equations show that ${^{1}1}$ admits two
additive inverses. Here, an additive inverse of $x$
is any element $y$ satisfying
\[
    x\oplus y=y\oplus x={^{0}0}.
\]

The existence of multiple additive inverses is
incompatible with full associativity. For example,
\begin{align*}
    ({^{1}1}\oplus{^{2}1})\oplus{^{3}1}
        &= {^{3}1},\\
    {^{1}1}\oplus({^{2}1}\oplus{^{3}1})
        &= {^{1}1}.
\end{align*}

However, addition remains associative when the
operands involve at most two distinct nonzero signs.
This follows from the additive identity property and
\emph{signed-associativity}, which we define in
Definition~\ref{def_signed_semigroup}.

For real nonnegative magnitudes, this additive rule
coincides with the binary operation on metric $s$-spiders
studied by Basrak \cite[p.~1139, Eq.~(4)]{Basrak2010}.
Basrak identifies its commutativity, non-associativity
for $s>2$, and multiple additive inverses in the context
of inductive means on metric trees. We extend the
additive construction to ordered cancellative magnitude
monoids admitting differences between comparable elements.
Under additional multiplicative assumptions, we equip the
same set with an internal multiplication and establish its
compatibility with addition.

We call the resulting algebraic framework \emph{Multisign Algebra}.
Unlike algebraic structures with multivalued addition,
this construction defines addition as a single-valued
binary operation. Multiple additive inverses result
from cancellation between distinct signs of equal
magnitude.

We make the following contributions:

\begin{enumerate}

    \item We extend the additive construction on real-valued
    metric spiders to abstract ordered magnitude monoids and
    define an internal multiplication on the resulting numerical
    system under the full magnitude assumptions.

    \item We establish signed-associativity and prove that the
    $s-1$ bilateral additive inverses of each nonzero element
    attain the upper bound for $n$i-signed-groups on the same
    canonically labelled multisign set with zero as identity.
    We characterize the maximal associative additive subsemigroups
    as copies of the classical Grothendieck group of the magnitude
    monoid and determine the solutions of additive equations
    of the form $x\oplus u=v$.

    \item We prove that multiplication satisfies
    associativity, commutativity, and bilateral
    distributivity over addition, and that every
    nonzero element has a multiplicative inverse.

    \item We characterize the group of nonzero elements
    under multiplication through the isomorphism
    \[
        \bigl(            \sigmac^{s}\setminus\{{^{0}0}\},
            \otimes
        \bigr)
        \cong
        \mathbb{Z}/s\mathbb{Z}\times P^{\times},
    \]
    where $P^{\times}=P\setminus\{0\}$ is the
    multiplicative group of nonzero magnitudes. We derive
    results on multiplicative orders, roots, mixed equations,
    and additive automorphisms induced by multiplication.

    \item We define algebraic structures with prescribed
    inverse counts and restricted associativity, and derive
    properties from their axioms.

    \item We characterize magnitude spaces satisfying
    the full magnitude assumptions as nonnegative cones
    of ordered fields, recover the corresponding ordered
    field for $s=2$, and characterize classical subfields
    according to the parity of $s$.

\end{enumerate}


Section~\ref{related_concepts} reviews related
algebraic structures.
Section~\ref{Generalization_concepts} defines the
multisign construction and establishes its properties.
Section~\ref{real_numbers} establishes its relationship
with ordered fields, including the real field.
Section~\ref{conclusions} presents the conclusions
and future work.

\section{Preliminaries}\label{related_concepts}

Let $A$ be a set equipped with an associative binary
operation $\circ$ and an identity element $e$.
If $x^{\prime}$ and $x^{\prime\prime}$ are bilateral
inverses of $x\in A$, then
\begin{align}
    x^{\prime}
    &= x^{\prime}\circ e \nonumber\\
    &= x^{\prime}\circ(x\circ x^{\prime\prime}) \nonumber\\
    &= (x^{\prime}\circ x)\circ x^{\prime\prime}
    \label{associative_step}\\
    &= e\circ x^{\prime\prime} \nonumber\\
    &= x^{\prime\prime}. \nonumber
\end{align}

Equation~\eqref{associative_step} uses associativity.
Therefore, an element cannot admit multiple bilateral
inverses in an associative structure with an identity.
The converse does not hold: non-associativity alone
does not imply the existence of multiple inverses.

The closest additive antecedent considered here is the
operation on metric spiders used by Basrak
\cite[Section~2, Eq.~(4)]{Basrak2010} to study inductive
means on metric trees. A finite $s$-spider consists of
$s$ copies of $\mathbb{R}_{\geq 0}$ whose origins are
identified. Write its origin as $o$ and a nonzero point
as $(i,a)$, where $i\in\{1,\ldots,s\}$ and $a>0$.
Writing $\oplus_N$ for the spider operation, we have
\[
    (i,a)\oplus_N(j,b)=
    \begin{cases}
        (i,a+b), & i=j,\\
        (i,a-b), & i\ne j,\ a>b,\\
        (j,b-a), & i\ne j,\ a<b,\\
        o, & i\ne j,\ a=b.
    \end{cases}
\]
The origin is an identity. The correspondence
$o\mapsto {^{0}0}$ and $(i,a)\mapsto {^{i}a}$ is an
isomorphism of the spider's additive structure with
$(\sigmac^{s}_{\mathbb{R}_{\geq 0}},\oplus)$ under the
addition defined in Section~\ref{Generalization_concepts}.
Basrak explicitly identifies the $s-1$ additive inverses
of each nonzero point and the failure of associativity
for $s>2$. He also uses scaling by nonnegative real
numbers, but does not define an internal multiplication
between spider elements. We distinguish the present
construction by its abstract magnitude assumptions,
its sign-labelled axioms, and its internal multiplication
with bilateral distributivity.

The classical group completion of a commutative monoid
provides another relevant antecedent. As described by
Weibel \cite[Chapter~II, Section~1]{Weibel2013}, a
commutative monoid $P$ admits a canonical homomorphism
into its Grothendieck group $G(P)$, universal among
monoid homomorphisms from $P$ to abelian groups.
This homomorphism is injective precisely when $P$ is
cancellative. For a cancellative commutative monoid,
the standard pair construction gives
\[
    G(P)=(P\times P)/\sim,\qquad
    (a,b)\sim(c,d)\iff a+d=b+c.
\]
Group completion yields an associative operation and
unique additive inverses. Under our additive magnitude
assumptions, the two-sign construction recovers this
completion. For $s>2$, the full multisign additive
structure is not an abelian group and therefore cannot
itself be a classical group completion. Instead, we
show that its maximal associative additive subsemigroups
are precisely the two-sign subsets, each isomorphic to
$G(P)$. The group completion itself is classical; the
result concerns how its copies occur within the
multisign structure.

Quasigroups and loops require unique solutions to the
division equations
\[
    a\circ x=b,
    \qquad
    y\circ a=b.
\]
A loop additionally requires an identity element.
Consequently, each element of a loop admits a unique
left inverse and a unique right inverse, which need
not coincide \cite{Figula2020,Vishne2021}.
These conditions differ from the left and right
inverse properties \cite{Vishne2021}.
Equational division quasigroups also require uniquely
solvable division equations, without necessarily
requiring an identity \cite{Im2021,Im2024}.
Evans studies ordered structures on loops and quasigroups
\cite{Evans1970}.

The multisign construction relaxes the uniqueness
requirement for additive division. For $s>2$, the
equation
\[
    x\oplus y={^{0}0}
\]
admits $s-1$ distinct solutions for each nonzero $x$.
Since addition is commutative, these solutions are
bilateral additive inverses.

Thus, $(\sigmac^{s},\oplus)$ is neither a quasigroup
nor a loop in the classical sense when $s>2$.
We introduce the notion of an $n$i-loop to describe
structures with an identity and a prescribed number
$n$ of bilateral inverses for each nonidentity
element.

Other non-associative structures satisfy identities
that replace or restrict classical associativity.
Lie and Leibniz algebras satisfy identities involving
their bracket operations \cite{Iachello2014,Ayupov1997}.
Jordan, Malcev, and alternative algebras, as well as
Moufang loops, impose other algebraic identities
\cite{Jordan1933,Moufang1935,Casas2019}.

The Bol--Moufang identities impose restrictions on
quasigroups and loops \cite{Kunen1996,Phillips2010}.
Pre-Lie algebras impose symmetry conditions on the
associator \cite{Chapoton2013}, while distributive
quasigroups satisfy self-distributivity
\cite{Jedlicka2017}.
Smooth loopoids and quasiloopoids extend related
concepts to geometric settings \cite{Grabowski2022}.

In contrast, signed-associativity requires the
associative law for triples whose operands involve
at most two distinct sign labels.
The condition depends on the sign labels of the
operands rather than imposing associativity on
arbitrary triples.
We formalize this property in
Definition~\ref{def_signed_semigroup}.

Hyperfields and hyperrings generalize algebraic
structures by replacing addition with a set-valued
operation \cite{Viro2010,Jun2018,Velrajan2010}.
In particular, the sign hyperfield
$\mathbb{S}=\{-1,0,1\}$ defines hyperaddition by
\begin{align*}
    0\boxplus a &= \{a\},\\
    1\boxplus 1 &= \{1\},\\
    (-1)\boxplus(-1) &= \{-1\},\\
    1\boxplus(-1) &= \{-1,0,1\}.
\end{align*}

Although hyperaddition may produce multiple results,
each element $a\in\mathbb{S}$ has a unique additive
inverse $b$ satisfying
\[
    0\in a\boxplus b.
\]

The multisign construction differs in both its
underlying set and its addition.
It associates each nonzero element with a sign and
a magnitude, and defines a single-valued operation
$\oplus$.
For $s>2$, the multiplicity occurs in the solutions
to the additive inverse equation, not in the result
of the operation.

Related frameworks, such as tracts and partial
hyperstructures, describe algebraic dependence through
formal null relations \cite{Baker2019}.
These frameworks also support extensions involving
vectors, covectors, and skew hyperfields
\cite{Anderson2019,Bowler2026}.

H.~von~Eitzen describes PolySign Numbers ($P_s$)
\cite{vonEitzen2009}, which T.~Golden discusses
in the sci.math Usenet group \cite{polysignUsenet}.

PolySign Numbers represent elements through magnitudes
associated with several signs, subject to the
cancellation relation
\[
    \sum_{i=1}^{s}\sigma_i x=0,
\]
where $\sigma_i$ denotes a sign and $x$ a common
magnitude.

Addition combines the magnitudes associated with
each sign:
\[
    \sum_{i=1}^{s}\sigma_i x_i
    +
    \sum_{i=1}^{s}\sigma_i y_i
    =
    \sum_{i=1}^{s}\sigma_i(x_i+y_i).
\]

The cancellation relation allows the removal of a
common magnitude from all sign components.
Multiplication combines sign components according
to a cyclic rule.
The resulting addition is associative, and each
element has a unique additive inverse.

In contrast, each nonzero multisign number has
exactly one sign and one magnitude.
Additive cancellation occurs between any two
distinct signs of equal magnitude, independently
of the remaining sign labels.

For example, when $s=3$,
\[
    {^{1}1}\oplus{^{2}1}
    =
    {^{1}1}\oplus{^{3}1}
    =
    {^{0}0}.
\]

This property distinguishes multisign addition from
the cancellation relation used in PolySign Numbers.
Although both constructions use a cyclic rule for
multiplying signs, their additive operations and
inverse structures differ.

The multisign construction therefore combines
single-valued addition, a unique additive identity,
multiple additive inverses for $s>2$, and
signed-associativity.
It also admits an associative and commutative
multiplication that distributes over addition.

We introduce the following algebraic structures
to describe these properties independently of
the concrete multisign operations.

\begin{remark}
In an $n$i-loop or $n$i-signed-group, $n$ denotes
the prescribed number of bilateral inverses of each
nonidentity element. In $n$i-signed-rings, fields,
modules, vector spaces, and algebras, $n$ refers to
the additive operation.
For $(\sigmac^{s},\oplus)$, this parameter is
$n=s-1$. For a one-element structure, the condition
on nonidentity elements is vacuous; in that case,
$n$ is a specified parameter rather than an invariant
uniquely determined by the operation.
\end{remark}

\section{Generalization of the concept of sign and new algebraic structures}\label{Generalization_concepts}

\subsection{Sign of a number}

\begin{definition}\label{semiring_with_invertible_multiplication}
A \textbf{semiring with invertible multiplication}
is a nonempty set $P$ equipped with two binary
operations $+$ and $\cdot$ such that:

\begin{enumerate}[label=\textnormal{(\roman*)}]

    \item $(P,+)$ is a commutative monoid with
    additive identity $0$;

    \item $(P\setminus\{0\},\cdot)$ is an abelian
    group with multiplicative identity $1\neq 0$;

    \item multiplication distributes over addition:
    for all $a,b,c\in P$,
    \[
        a\cdot(b+c)=a\cdot b+a\cdot c
    \]
    and
    \[
        (b+c)\cdot a=b\cdot a+c\cdot a;
    \]

    \item the additive identity is absorbing under
    multiplication: for all $a\in P$,
    \[
        a\cdot 0=0\cdot a=0.
    \]

\end{enumerate}
\end{definition}

\begin{remark}
This structure is a commutative semifield in the
semiring sense. We use the term semiring with
invertible multiplication to distinguish it from
the generalized field structures introduced later.
\end{remark}

\begin{assumption}\label{ass:P-ordered}
Unless stated otherwise, we assume that $P$ is a
semiring with invertible multiplication satisfying
the following additional conditions on its additive
monoid and order:

\begin{enumerate}[label=\textnormal{(\roman*)}]

    \item $P$ has a total order $\leq$ with $0$
    as its least element;

    \item $(P,+)$ is cancellative: for all
    $a,b,c\in P$,
    \[
        a+c=b+c\implies a=b;
    \]

    \item the order is compatible with addition:
    for all $a,b,c\in P$,
    \[
        a<b\implies a+c<b+c;
    \]

    \item for every $a,b\in P$ such that $a\geq b$,
    there exists $c\in P$ satisfying
    \[
        b+c=a.
    \]

\end{enumerate}
\end{assumption}

By additive cancellation, condition (iv) determines
$c$ uniquely. We denote it by $a-b$. Conditions
(i)--(iv) also apply to additive monoids without
a multiplication; in that setting, we do not require
Definition~\ref{semiring_with_invertible_multiplication}.

\begin{proposition}\label{ass:P-no-additive-inverses}
For any commutative monoid $(P,+)$ satisfying
conditions (i)--(iv) of Assumption~\ref{ass:P-ordered},
no nonzero element has an additive inverse. That is, for all $a,b\in P$,
\[
    a+b=0\implies a=b=0.
\]
\end{proposition}

\begin{proof}
Suppose that $a+b=0$ and $a\neq 0$. Since $0$ is
the least element of $P$, we have $a>0$.
By compatibility of the order with addition,
\[
    a+b>0+b=b\geq 0,
\]
which contradicts $a+b=0$.
Therefore, $a=0$, and consequently $b=0$.
\end{proof}

\begin{remark}\label{remark:magnitude}
We interpret $P$ as a magnitude space. Its elements
represent magnitudes without signs. The absence
of nonzero additive inverses in $P$ separates
cancellation in the multisign construction from
addition within the magnitude space.
\end{remark}

\begin{remark}\label{example_semiring_with_invertible_multiplication}
The sets
\[
    \mathbb{R}_{\geq 0}
    \quad\text{and}\quad
    \mathbb{Q}_{\geq 0}
\]
with their usual addition, multiplication, and
order satisfy
Definition~\ref{semiring_with_invertible_multiplication}
and Assumption~\ref{ass:P-ordered}.

More generally, let $K$ be an ordered field.
Its nonnegative cone
\[
    K_{\geq 0}=\{a\in K:a\geq 0\}
\]
satisfies the same conditions.

Indeed, addition and multiplication preserve
nonnegativity, every positive element has a
positive multiplicative inverse, and $a-b$ is
nonnegative whenever $a\geq b$.

For example, consider the ordered field
\[
    K=\mathbb{Q}(\sqrt{2})
    =\{a+b\sqrt{2}:a,b\in\mathbb{Q}\},
\]
with the order inherited from $\mathbb{R}$.
Its nonnegative cone
\[
    P=
    \{a+b\sqrt{2}:a,b\in\mathbb{Q},
    \ a+b\sqrt{2}\geq 0\}
\]
is also an admissible magnitude space.

\end{remark}

\begin{definition}\label{sign}
Fix an integer $s\geq 2$.
We define the set of signs as
\[
    D=\{0,1,2,\ldots,s\}
    \subseteq\mathbb{N}_0,
\]
where we reserve $0$ for the additive identity
of the multisign construction and
\[
    D\setminus\{0\}=\{1,2,\ldots,s\}
\]
is the set of signs associated with nonzero
magnitudes.
\end{definition}

The elements of $D$ serve as labels. Their
numerical order does not define an order among
the corresponding signs.

\subsection[Sigma set]{\texorpdfstring{Sigma set ($\sigmac^{s}$)}{Sigma set}}

\begin{definition}\label{sign_num}
Let $P$ be a nontrivial commutative monoid under addition
with identity $0$, equipped with an order and satisfying
conditions (i)--(iv) of Assumption~\ref{ass:P-ordered}.
Let $D=\{0,1,\ldots,s\}$ be the set of signs defined in
Definition~\ref{sign}, with $s\geq 2$. For the additive construction we require only these
monoid and order conditions. For the full construction
with $\otimes$, we also require
Definition~\ref{semiring_with_invertible_multiplication}.

We define the \textbf{multisign set} as
\[
    \sigmac_P^{s}
    =
    \{{}^{0}0\}
    \cup
    \left\{
        {}^{d}p :
        d\in D\setminus\{0\},
        \ p\in P\setminus\{0\}
    \right\}.
\]

Each element ${}^{d}p$ represents a formal pair
$(d,p)$, where $d$ is its sign and $p$ its magnitude.
Two multisign numbers are equal if and only if
their signs and magnitudes coincide:
\[
    {}^{i}a={}^{j}b
    \iff i=j\ \land\ a=b.
\]

When the magnitude space $P$ is fixed, we write
$\sigmac^{s}$ instead of $\sigmac_P^{s}$.
\end{definition}

\begin{remark}\label{def:zero-sign}
By Definition~\ref{sign_num}, every
${}^{d}p\in\sigmac^{s}$ satisfies
\[
    d=0
    \iff p=0
    \iff {}^{d}p={}^{0}0.
\]

Thus, ${}^{0}0$ is the unique element with zero
magnitude and the unique element with sign zero.
Neither ${}^{0}p$ with $p\neq 0$ nor ${}^{d}0$
with $d\neq 0$ belongs to $\sigmac^{s}$.
\end{remark}

\begin{remark}
The sign component $d$ is a discrete label rather
than an element of the magnitude space $P$.
Consequently, the operations defined on $P$ do not
directly apply to sign labels.

The operations on $\sigmac^{s}$ determine the
relations between signs, rather than polynomial
relations on elements of $P$.
\end{remark}

\begin{example}
For $P=\mathbb{R}_{\geq 0}$ and $s=3$, the
multisign set is
\[
    \sigmac_{\mathbb{R}_{\geq 0}}^{3}
    =
    \{{}^{0}0\}
    \cup
    \bigcup_{d=1}^{3}
    \left\{
        {}^{d}p:p\in\mathbb{R}_{>0}
    \right\}.
\]

For any $p>0$, the elements
\[
    {}^{1}p,\quad {}^{2}p,\quad {}^{3}p
\]
have the same magnitude but are distinct.
\end{example}

\subsection[ni-Loop]{\texorpdfstring{$n$i-Loop}{ni-Loop}}

For the abstract signed structures below, a sign-labelled
set is a nonempty set $X$ equipped with a map
$\ell_X:X\to D_X$ into a set of sign labels $D_X$.
The map need not be injective or surjective. We fix
the labelling as part of each signed structure.
Signed-associativity depends on this choice: an injective
labelling restricts the condition to triples with at most
two distinct elements, whereas a labelling with at most
two labels in its image requires full associativity. In the
multisign set, $\ell_{\sigmac^{s}}({}^{d}p)=d$ and
$D_{\sigmac^{s}}=D$. We call this the canonical sign labelling.
An abstract carrier need not itself be a multisign set.
We use automorphism of an operation to mean a
bijection preserving that operation. We do not require
it to preserve sign labels unless we state otherwise.

\begin{definition}\label{def:ni-loop}
Let $n\in\mathbb{N}^{+}$. An \textbf{$n$i-loop}
$(X,\circledast)$ is a nonempty set
equipped with a binary operation $\circledast$
satisfying the following properties:

\begin{enumerate}[label=\textnormal{\arabic*.}]

\item \textnormal{\textbf{Internal operation}}:
For all $x,y\in X$,
\[
    x\circledast y\in X.
\]

\item \textnormal{\textbf{Existence of identity}}:
There exists an element $e_{\circledast}\in X$
such that, for every $x\in X$,
\[
    e_{\circledast}\circledast x
    =x
    =x\circledast e_{\circledast}.
\]

\item \textnormal{\textbf{Existence of inverse elements}}:
For every $x\in X\setminus
\{e_{\circledast}\}$, there exist exactly $n$
distinct elements $y\in X$ satisfying
\[
    x\circledast y
    =e_{\circledast}
    =y\circledast x.
\]

\end{enumerate}
\end{definition}

For any set $X$ equipped with a binary operation
$\circledast$ and a two-sided identity $e_{\circledast}$,
we define the set of bilateral inverses of each $x\in X$ as
\[
    I_{\circledast}(x)
    =
    \left\{
        y\in X:
        x\circledast y
        =
        y\circledast x
        =
        e_{\circledast}
    \right\}.
\]

\begin{proposition}
The identity of an $n$i-loop is unique and is
its own unique bilateral inverse.
\end{proposition}

\begin{proof}
Suppose that $e$ and $f$ are identity elements.
Then
\[
    e=e\circledast f=f.
\]

Moreover, if $y$ is an inverse of
$e_{\circledast}$, then
\[
    y=e_{\circledast}\circledast y
    =e_{\circledast}.
\]
\end{proof}

Consequently,
\[
    |I_{\circledast}(x)|
    =
    \begin{cases}
        1, & x=e_{\circledast},\\
        n, & x\neq e_{\circledast}.
    \end{cases}
\]

\begin{proposition}
An $n$i-loop with $n>1$ and at least one
nonidentity element is not associative.
\end{proposition}

\begin{proof}
Let $x\neq e_{\circledast}$ and let
$y,z\in I_{\circledast}(x)$ with $y\neq z$.

Then
\[
    (y\circledast x)\circledast z
    =
    e_{\circledast}\circledast z
    =
    z,
\]
whereas
\[
    y\circledast(x\circledast z)
    =
    y\circledast e_{\circledast}
    =
    y.
\]

Since $y\neq z$, associativity fails.
\end{proof}

\begin{remark}
Unlike a classical loop, an $n$i-loop does not require
existence or uniqueness of solutions
to the division equations
\[
    a\circledast x=b,
    \qquad
    y\circledast a=b.
\]

For $n>1$, uniqueness fails when
$a\neq e_{\circledast}$ and $b=e_{\circledast}$.

The case $n=1$ guarantees unique bilateral
inverses but does not imply the division
properties of a classical loop.
\end{remark}

\subsection{Signed-Semigroup}\label{signed_semigroup}

\begin{definition}\label{def_signed_semigroup}
A \textbf{signed-semigroup}
$(X,\circledast)$ is a sign-labelled set
equipped with a binary operation $\circledast$
satisfying the following properties:

\begin{enumerate}[label=\textnormal{\arabic*.}]

\item \textnormal{\textbf{Internal operation}}:
For all $x,y\in X$,
\[
    x\circledast y\in X.
\]

\item \textnormal{\textbf{Signed-Associativity}}
\label{s_almost_associativity_2}:
For all $x,y,z\in X$,
\[
    (x\circledast y)\circledast z
    =x\circledast(y\circledast z)
\]
whenever
\[
    |\{\ell_X(x),\ell_X(y),\ell_X(z)\}|\leq 2.
\]

\end{enumerate}
\end{definition}

\begin{remark}
Signed-associativity restricts the associative law
to triples involving at most two distinct sign
labels. It does not exclude associativity for
triples involving three distinct signs.

Every associative binary operation on
a sign-labelled set $X$ satisfies signed-associativity,
but the converse does not hold in general.

Unlike classical associativity, signed-associativity
depends on both the binary operation and the
sign labels of its operands.
\end{remark}

\begin{remark}\label{assoc-with-identity}
If $\circledast$ admits an identity element
$e_{\circledast}$, then associativity holds
whenever one of the operands is the identity.

For all $x,y\in X$,
\begin{align*}
    (e_{\circledast}\circledast x)\circledast y
    &=
    e_{\circledast}\circledast(x\circledast y)
    =
    x\circledast y,\\
    (x\circledast e_{\circledast})\circledast y
    &=
    x\circledast(e_{\circledast}\circledast y)
    =
    x\circledast y,\\
    (x\circledast y)\circledast e_{\circledast}
    &=
    x\circledast(y\circledast e_{\circledast})
    =
    x\circledast y.
\end{align*}

Therefore, in a signed-semigroup with an identity,
associativity can fail only for triples of
nonidentity elements with three distinct signs.
\end{remark}

\subsection{Signed-Monoid}

\begin{definition}
A \textbf{signed-monoid}
$(X,\circledast)$ is a signed-semigroup
with an identity element
$e_{\circledast}\in X$ such that,
for every $x\in X$,
\[
    e_{\circledast}\circledast x
    =
    x
    =
    x\circledast e_{\circledast}.
\]
\end{definition}

\begin{remark}
Every classical monoid equipped with a sign labelling
is a signed-monoid.

A signed-monoid is a classical monoid if and only
if its binary operation is fully associative.

Signed-associativity alone does not guarantee
that arbitrary products are independent of
their parenthesization.
\end{remark}

\subsection[ni-Signed-Group]{\texorpdfstring{$n$i-Signed-Group}{ni-Signed-Group}}

\begin{definition}
An \textbf{$n$i-signed-group}
$(X,\circledast)$ is an $n$i-loop on a sign-labelled set
satisfying signed-associativity
in the sense of
Definition~\ref{def_signed_semigroup}.

Equivalently, an $n$i-signed-group is a
signed-monoid in which every nonidentity
element has exactly $n$ bilateral inverses.

We call an $n$i-signed-group
\textbf{abelian} if, for all $x,y\in X$,
\[
    x\circledast y=y\circledast x.
\]
\end{definition}

\begin{remark}
Every classical group equipped with a sign labelling
is a $1$i-signed-group.

An $n$i-signed-group is a classical group
if and only if its operation is fully
associative. In this case, the uniqueness
of inverses follows from associativity,
and $n=1$ whenever $X$ has a nonidentity element.

The condition $n=1$ alone does not imply
full associativity.
\end{remark}

\begin{proposition}\label{prop:inverse-signs}
Let $(\sigmac^{s},\circledast)$ be an
$n$i-signed-group with the canonical sign labelling, and let
$x={}^{i}a\neq e_{\circledast}$.

If $y={}^{j}b$ and $z={}^{k}c$ are distinct
bilateral inverses of $x$, then
\[
    |\{i,j,k\}|=3.
\]
\end{proposition}

\begin{proof}
Since $y,z\in I_{\circledast}(x)$,
\begin{align*}
    (y\circledast x)\circledast z
    &=
    e_{\circledast}\circledast z
    =
    z,\\
    y\circledast(x\circledast z)
    &=
    y\circledast e_{\circledast}
    =
    y.
\end{align*}

As $y\neq z$, associativity fails for
the triple $(y,x,z)$.

By signed-associativity, this is possible
only if the three sign labels are distinct.
Hence,
\[
    |\{i,j,k\}|=3.
\]
\end{proof}

\begin{corollary}\label{cor:inverse-sign-distinction}
Let $(\sigmac^{s},\circledast)$ be an
$n$i-signed-group with the canonical sign labelling and $n>1$.

For every nonidentity element $x$, the
elements of $I_{\circledast}(x)$ have
pairwise distinct signs, each different
from the sign of $x$.
\end{corollary}

\begin{proof}
Since $n>1$, for each inverse we can choose
another distinct inverse.

The result follows from
Proposition~\ref{prop:inverse-signs}.
\end{proof}

\begin{corollary}\label{cor:inverse-sign-bound}
Let $(\sigmac^{s},\circledast)$ be an
$n$i-signed-group with the canonical sign labelling and $n>1$ and at least
one nonidentity element. Then
\[
    n\leq s.
\]

If $e_{\circledast}={}^{0}0$, then
\[
    n\leq s-1.
\]
\end{corollary}

\begin{proof}
The sign set $D$ contains $s+1$ labels.
By Corollary~\ref{cor:inverse-sign-distinction},
the inverses of a nonidentity element $x$
have pairwise distinct signs, excluding
the sign of $x$. Therefore,
\[
    n\leq |D|-1=s.
\]

Suppose now that $e_{\circledast}={}^{0}0$.
No inverse of a nonidentity element can
be the identity itself.

Since ${}^{0}0$ is the unique element with
sign zero, neither sign zero nor the sign of $x$
can occur among the inverse signs. Consequently,
\[
    n\leq |D|-2=s-1.
\]
\end{proof}

\begin{definition}
Let $\sigmac^{s}$ be the multisign set
defined in Definition~\ref{sign_num}.

We define the absolute value function
\[
    |\cdot|:\sigmac^{s}\longrightarrow P
\]
by
\[
    |{}^{d}p|=p.
\]
\end{definition}

\begin{definition}\label{addition}
Let $\sigmac^{s}$ be the multisign set defined in
Definition~\ref{sign_num}.
We define the addition
\[
    \oplus:\sigmac^{s}\times\sigmac^{s}
    \longrightarrow\sigmac^{s}
\]
by
\[
{^{i}a}\oplus{^{j}b}
=
\begin{cases}
    {^{i}(a+b)}, & i=j,\\
    {^{i}(a-b)}, & i\neq j,\ a>b,\\
    {^{j}(b-a)}, & i\neq j,\ a<b,\\
    {^{0}0}, & i\neq j,\ a=b.
\end{cases}
\]

Here, $+$ denotes addition in $P$, and the
differences $a-b$ and $b-a$ follow the convention
in Assumption~\ref{ass:P-ordered}.
\end{definition}

\begin{proposition}\label{prop:addition-basic}
The addition $\oplus$ is closed and commutative
on $\sigmac^{s}$, with identity ${^{0}0}$.
\end{proposition}

\begin{proof}
If $i=j\neq 0$, then $a,b>0$, and
Proposition~\ref{ass:P-no-additive-inverses}
ensures that $a+b\neq 0$.
If $i=j=0$, both operands are ${^{0}0}$,
and their sum is ${^{0}0}$.

If $i\neq j$ and $a>b$, then $a-b>0$ and
$i\neq 0$. Hence ${^{i}(a-b)}\in\sigmac^{s}$.
The case $b>a$ follows analogously.
If $i\neq j$ and $a=b$, the result is ${^{0}0}$.
Thus, $\oplus$ is closed on $\sigmac^{s}$.

Commutativity follows from the commutativity
of addition in $P$ and the symmetry of the
remaining cases under the exchange of operands.

Finally, for every ${^{i}a}\in\sigmac^{s}$,
\[
    {^{i}a}\oplus{^{0}0}
    =
    {^{0}0}\oplus{^{i}a}
    =
    {^{i}a}.
\]
Therefore, ${^{0}0}$ is the additive identity.
\end{proof}

\subsection{Additive inverses and signed-associativity}

\begin{proposition}\label{prop:additive-inverses}
For every nonzero element
${^{i}a}\in\sigmac^{s}$, its set of bilateral
additive inverses is
\[
    I_{\oplus}({^{i}a})
    =
    \left\{
        {^{j}a}:
        j\in D\setminus\{0,i\}
    \right\}.
\]

Consequently,
\[
    |I_{\oplus}({^{i}a})|=s-1.
\]

Moreover,
\[
    I_{\oplus}({^{0}0})=\{{^{0}0}\}.
\]
\end{proposition}

\begin{proof}
Let ${^{i}a}\neq{^{0}0}$.
By Proposition~\ref{prop:addition-basic},
none of its additive inverses can be ${^{0}0}$.

For any nonzero ${^{j}b}\in\sigmac^{s}$,
Definition~\ref{addition} gives
\[
    {^{i}a}\oplus{^{j}b}={^{0}0}
    \iff i\neq j\ \land\ a=b.
\]

There are exactly $s-1$ nonzero sign labels
different from $i$.
Since addition is commutative, all these
inverses are bilateral.

The identity has itself as its unique
additive inverse.
\end{proof}

For example, take $a\in P\setminus\{0\}$. When $s=3$,
\[
    I_{\oplus}({^{1}a})
    =
    \{{^{2}a},{^{3}a}\},
\]
whereas for $s=4$,
\[
    I_{\oplus}({^{1}a})
    =
    \{{^{2}a},{^{3}a},{^{4}a}\}.
\]

\begin{corollary}\label{cor:addition-not-associative}
For $s>2$, the addition $\oplus$ is not
associative on $\sigmac^{s}$.
\end{corollary}

\begin{proof}
Let $i,j,k$ be three distinct nonzero signs
and let $a>0$.
Then
\begin{align*}
    ({^{i}a}\oplus{^{j}a})\oplus{^{k}a}
    &={^{k}a},\\
    {^{i}a}\oplus({^{j}a}\oplus{^{k}a})
    &={^{i}a}.
\end{align*}

Since $i\neq k$, these results are distinct.
\end{proof}

\begin{definition}\label{def:two-sign-subset}
For distinct signs $i,j\in D\setminus\{0\}$,
we define
\[
    S_{ij}
    =
    \{{^{0}0}\}
    \cup
    \left\{{^{i}a}:a\in P\setminus\{0\}\right\}
    \cup
    \left\{{^{j}a}:a\in P\setminus\{0\}\right\}.
\]
\end{definition}

\begin{proposition}\label{prop:two-sign-group}
For every pair of distinct nonzero signs $i,j$,
the structure $(S_{ij},\oplus)$ is an abelian
group isomorphic to the Grothendieck group
of the additive monoid $(P,+)$.
\end{proposition}

\begin{proof}
The set $S_{ij}$ is closed under $\oplus$: a sum
of its elements either retains one of the signs
$i,j$ or equals ${^{0}0}$, by Definition~\ref{addition}.
Since $(P,+)$ is a cancellative commutative
monoid, its Grothendieck group admits the
standard pair construction
\cite[Chapter~II, Section~1]{Weibel2013}:
\[
    G(P)=(P\times P)/{\sim},
\]
where
\[
    (a,b)\sim(c,d)
    \iff a+d=b+c.
\]

Writing $[a,b]$ for the equivalence class
of $(a,b)$, we define addition in $G(P)$ by
\[
    [a,b]+[c,d]=[a+c,b+d].
\]

The identity is $[0,0]$, and the inverse
of $[a,b]$ is $[b,a]$.

By Assumption~\ref{ass:P-ordered}, every
class admits one of the representations
\[
    [a,b]
    =
    \begin{cases}
        [a-b,0], & a>b,\\
        [0,b-a], & a<b,\\
        [0,0], & a=b.
    \end{cases}
\]

Define
\[
    \phi_{ij}:S_{ij}\longrightarrow G(P)
\]
by
\[
    \phi_{ij}(x)
    =
    \begin{cases}
        [0,0], & x={^{0}0},\\
        [a,0], & x={^{i}a},\ a>0,\\
        [0,a], & x={^{j}a},\ a>0.
    \end{cases}
\]

The preceding representations establish
surjectivity.

For $a,b>0$, equality $[a,0]=[b,0]$
implies $a=b$, and equality
$[a,0]=[0,b]$ would imply $a+b=0$,
which contradicts
Proposition~\ref{ass:P-no-additive-inverses}.
The same argument applies to classes of
the form $[0,a]$.
Neither form represents $[0,0]$ when its magnitude is positive.
Hence, $\phi_{ij}$ is injective.

For operands with the same sign,
\begin{align*}
    \phi_{ij}({^{i}a}\oplus{^{i}b})
    &=[a+b,0]\\
    &=[a,0]+[b,0],
\end{align*}
and
\begin{align*}
    \phi_{ij}({^{j}a}\oplus{^{j}b})
    &=[0,a+b]\\
    &=[0,a]+[0,b].
\end{align*}

For operands with different signs,
\[
    \phi_{ij}({^{i}a})
    +
    \phi_{ij}({^{j}b})
    =
    [a,b].
\]

The representation of $[a,b]$ according
to the order of $a$ and $b$ gives
\[
    \phi_{ij}({^{i}a}\oplus{^{j}b})
    =
    \phi_{ij}({^{i}a})
    +
    \phi_{ij}({^{j}b}).
\]

The case involving ${^{0}0}$ follows
from the identity property.

Thus, $\phi_{ij}$ is a bijective
homomorphism, and
\[
    (S_{ij},\oplus)\cong(G(P),+).
\]

Since $G(P)$ is an abelian group,
so is $(S_{ij},\oplus)$.
\end{proof}

\begin{corollary}\label{cor:signed-associativity-addition}
The addition $\oplus$ satisfies
signed-associativity.
\end{corollary}

\begin{proof}
Let ${^{i}a},{^{j}b},{^{k}c}\in\sigmac^{s}$.

If one of the operands is ${^{0}0}$,
associativity follows from
Remark~\ref{assoc-with-identity}.

Otherwise, if $|\{i,j,k\}|\leq 2$,
all three operands belong to a common
two-sign subset $S_{uv}$.

By Proposition~\ref{prop:two-sign-group},
$(S_{uv},\oplus)$ is an abelian group.
Therefore,
\[
    ({^{i}a}\oplus{^{j}b})\oplus{^{k}c}
    =
    {^{i}a}\oplus({^{j}b}\oplus{^{k}c}).
\]
\end{proof}

\begin{theorem}\label{s-abelian_group_proof}
The structure $(\sigmac^{s},\oplus)$ is an
abelian $(s-1)$i-signed-group with identity
${^{0}0}$.
\end{theorem}

\begin{proof}
By Proposition~\ref{prop:addition-basic},
addition is closed and commutative and
admits the identity ${^{0}0}$.

By Proposition~\ref{prop:additive-inverses},
every nonidentity element admits exactly
$s-1$ bilateral additive inverses.

By Corollary~\ref{cor:signed-associativity-addition},
addition satisfies signed-associativity.

Hence, $(\sigmac^{s},\oplus)$ satisfies
the axioms of an abelian $(s-1)$i-signed-group.
\end{proof}

\begin{corollary}\label{cor:maximal-inverse-count}
Among $n$i-signed-groups defined on
$\sigmac^{s}$ with its canonical sign labelling
and identity ${^{0}0}$,
the addition $\oplus$ attains the maximum
possible number of bilateral inverses
for each nonidentity element.
\end{corollary}

\begin{proof}
Corollary~\ref{cor:inverse-sign-bound}
establishes that $n\leq s-1$ for $n>1$.
The same bound holds for $n=1$ since $s\geq2$.

By Theorem~\ref{s-abelian_group_proof},
addition attains this bound with $n=s-1$.
\end{proof}

\begin{remark}\label{classical-s2-local}
For $s=2$, the entire multisign set
coincides with $S_{12}$.
Therefore,
\[
    (\sigmac^{2},\oplus)\cong(G(P),+),
\]
and addition is fully associative.

For $s>2$, full associativity fails,
although every two-sign subset remains
an abelian group.
\end{remark}

\begin{proposition}\label{prop:maximal-associative-subsets}
The sets $S_{ij}$, with $i$ and $j$
distinct nonzero signs, are exactly
the maximal associative subsemigroups
of $(\sigmac^{s},\oplus)$.
\end{proposition}

\begin{proof}
By Proposition~\ref{prop:two-sign-group},
each $S_{ij}$ is closed under addition
and associative.

Suppose that an associative subsemigroup
contains three elements with distinct
nonzero signs.
By reordering the elements, we may write
them as
\[
    x={^{i}a},\qquad
    y={^{j}b},\qquad
    z={^{k}c},
\]
where $a\geq b\geq c>0$ and $i,j,k$
are pairwise distinct.

If $b=c$, then $y\oplus z={^{0}0}$ and
$x\oplus(y\oplus z)=x$.
If $b>c$, then $y\oplus z={^{j}(b-c)}$ and
$a>b-c$. In both cases, uniqueness of differences
and associativity in $P$ give
\[
    x\oplus(y\oplus z)
    =
    {^{i}\bigl((a-b)+c\bigr)}.
\]

Since $a\geq b$, the sum $x\oplus y$ equals
${^{i}(a-b)}$ when $a>b$ and ${^{0}0}$ when $a=b$.
In either case,
\[
    (x\oplus y)\oplus z
    =
    \begin{cases}
        {^{i}\bigl((a-b)-c\bigr)},
            & a-b>c,\\
        {^{0}0},
            & a-b=c,\\
        {^{k}\bigl(c-(a-b)\bigr)},
            & a-b<c.
    \end{cases}
\]

If $a-b>c$, the two results have
different magnitudes, since $c>0$.

If $a-b=c$, the second expression is
zero and the first is nonzero.

If $a-b<c$, the two results have
different nonzero signs.

Thus,
\[
    (x\oplus y)\oplus z
    \neq
    x\oplus(y\oplus z),
\]
contradicting associativity.

Consequently, every associative
subsemigroup contains elements with at
most two distinct nonzero signs and
is contained in some $S_{ij}$.

Since each $S_{ij}$ is associative and
adding an element with a third nonzero sign
would violate associativity, the sets $S_{ij}$
are exactly the maximal associative
subsemigroups.
\end{proof}

\begin{corollary}\label{cor:number-associative-subsets}
The number of maximal associative
subsemigroups of $(\sigmac^{s},\oplus)$ is
\[
    \binom{s}{2}=\frac{s(s-1)}{2}.
\]

Each is an abelian group isomorphic
to $G(P)$.
\end{corollary}

\begin{proof}
Each maximal associative subsemigroup
corresponds to an unordered pair of
distinct nonzero signs.
The result follows from
Propositions~\ref{prop:two-sign-group}
and~\ref{prop:maximal-associative-subsets}.
\end{proof}

\begin{proposition}\label{prop:additive-equations}
For $u,v\in\sigmac^{s}$, define
\[
    E(u,v)
    =
    \{x\in\sigmac^{s}:x\oplus u=v\}.
\]

If $u={^{0}0}$, then
\[
    E(u,v)=\{v\}.
\]

If $u={^{i}a}\neq{^{0}0}$, then:

\begin{enumerate}[label=\textnormal{(\roman*)}]

    \item If $v={^{0}0}$,
    \[
        E(u,v)
        =
        \{{^{k}a}:k\in D\setminus\{0,i\}\}.
    \]

    \item If $v={^{i}b}$ with $b>a$,
    \[
        E(u,v)=\{{^{i}(b-a)}\}.
    \]

    \item If $v={^{i}a}$,
    \[
        E(u,v)=\{{^{0}0}\}.
    \]

    \item If $v={^{i}b}$ with $0<b<a$,
    \[
        E(u,v)
        =
        \left\{
            {^{k}(a-b)}:
            k\in D\setminus\{0,i\}
        \right\}.
    \]

    \item If $v={^{j}b}$ with $j\neq i$
    and $b>0$,
    \[
        E(u,v)=\{{^{j}(a+b)}\}.
    \]

\end{enumerate}
\end{proposition}

\begin{proof}
If $u={^{0}0}$, the result follows
from the identity property.

For $u\neq{^{0}0}$, case (i) follows
from Proposition~\ref{prop:additive-inverses}.

If $v={^{i}b}$ with $b>a$, the result
must arise from adding an element with
sign $i$. The following equation determines its magnitude uniquely:
\[
    a+|x|=b.
\]

If $v={^{i}a}$, any nonzero operand
would either change the magnitude or
change the sign of the result.
Therefore, $x={^{0}0}$ is the only
solution.

If $v={^{i}b}$ with $0<b<a$, the
solution must have a sign different
from $i$ and magnitude satisfying
\[
    a-|x|=b.
\]
Hence, $|x|=a-b$, with $s-1$ possible
nonzero sign labels.

Finally, if $v={^{j}b}$ with $j\neq i$,
the solution must have sign $j$ and
magnitude satisfying
\[
    |x|-a=b.
\]
Thus, $x={^{j}(a+b)}$ is the unique
solution.

Substituting the listed candidates into
Definition~\ref{addition} verifies that each is a
solution. These cases exhaust all possible values
of $u$ and $v$.
\end{proof}

\begin{corollary}\label{cor:additive-equation-cardinality}
Every equation
\[
    x\oplus u=v
\]
has a solution in $\sigmac^{s}$.

Moreover,
\[
    |E(u,v)|\in\{1,s-1\}.
\]
\end{corollary}

\begin{proof}
The result follows from
Proposition~\ref{prop:additive-equations},
since $s\geq2$.
\end{proof}

\begin{corollary}\label{cor:additive-cancellation}
The operation $\oplus$ satisfies additive
cancellation if and only if $s=2$.
\end{corollary}

\begin{proof}
For $s=2$,
Corollary~\ref{cor:additive-equation-cardinality}
ensures that every equation $x\oplus u=v$
has a unique solution.

Therefore,
\[
    x\oplus u=y\oplus u
    \implies x=y.
\]

For $s>2$, let $i,j,k$ be distinct
nonzero signs and let $a>0$.
Then
\[
    {^{j}a}\oplus{^{i}a}
    =
    {^{k}a}\oplus{^{i}a}
    =
    {^{0}0},
\]
although ${^{j}a}\neq{^{k}a}$.

Thus, additive cancellation fails
when $s>2$.
\end{proof}

For example, in $\sigmac^{4}$ over
$\mathbb{R}_{\geq0}$, the equation
\[
    x\oplus{^{1}5}={^{1}2}
\]
has exactly three solutions:
\[
    x\in\{{^{2}3},{^{3}3},{^{4}3}\}.
\]

In contrast, the equation
\[
    x\oplus{^{1}5}={^{2}2}
\]
has the unique solution
\[
    x={^{2}7}.
\]

Thus, the nonuniqueness of solutions
extends beyond the additive inverse
equation and depends on the signs and
magnitudes of the prescribed operands.

\begin{definition}\label{multiplication}
Let $P$ be a semiring with invertible multiplication
satisfying Assumption~\ref{ass:P-ordered}, and let
$\sigmac^{s}$ be the multisign set defined in
Definition~\ref{sign_num}.
We define the multiplication
\[
    \otimes:\sigmac^{s}\times\sigmac^{s}
    \longrightarrow\sigmac^{s}
\]
by
\[
{^{i}a}\otimes{^{j}b}
=
\begin{cases}
    {^{0}0}, & a=0\ \text{or}\ b=0,\\
    {^{k}(ab)}, & a,b\neq 0,
\end{cases}
\]
where
\[
    k=1+((i+j-2)\bmod s)
\]
and $ab$ denotes multiplication in $P$. Throughout,
$r\bmod s$ denotes the unique integer in
$\{0,\ldots,s-1\}$ congruent to $r$ modulo $s$,
including when $r$ is negative.
\end{definition}

For the remainder of Section~3, whenever we use $\otimes$
on the multisign set, we assume the full magnitude conditions
of Definition~\ref{semiring_with_invertible_multiplication}
and Assumption~\ref{ass:P-ordered}, unless stated otherwise.

\begin{remark}\label{remark:modular-signs}
For nonzero operands, multiplication combines
magnitudes through multiplication in $P$ and
sign labels through addition modulo $s$.

The identification
\[
    i\longmapsto[i-1]_s
\]
maps the nonzero sign labels to the elements of
the cyclic group $\mathbb{Z}/s\mathbb{Z}$,
where $[r]_s$ denotes the residue class of $r$
modulo $s$.
\end{remark}

\begin{proposition}\label{prop:multiplicative-direct-product}
Let
\[
    P^{\times}=P\setminus\{0\}.
\]
Then
\[
    \bigl(
        \sigmac^{s}\setminus\{{^{0}0}\},
        \otimes
    \bigr)
    \cong
    \mathbb{Z}/s\mathbb{Z}\times P^{\times}
\]
as abelian groups.
\end{proposition}

\begin{proof}
By Definition~\ref{semiring_with_invertible_multiplication},
$P^{\times}$ is an abelian group under multiplication.

Since $ab\neq 0$ whenever $a,b\neq 0$ and
\[
    1+((i+j-2)\bmod s)\in\{1,\ldots,s\},
\]
the nonzero multisign set is closed under
multiplication.

Define
\[
    \Psi:
    \sigmac^{s}\setminus\{{^{0}0}\}
    \longrightarrow
    \mathbb{Z}/s\mathbb{Z}\times P^{\times}
\]
by
\[
    \Psi({^{i}a})=([i-1]_s,a).
\]

The uniqueness of the sign and magnitude of each
nonzero element establishes that $\Psi$ is bijective.

The operation on the direct product is
\[
    ([r]_s,a)\ast([t]_s,b)
    =
    ([r+t]_s,ab).
\]

By Definition~\ref{multiplication},
\begin{align*}
    \Psi({^{i}a}\otimes{^{j}b})
    &=
    ([i+j-2]_s,ab)\\
    &=
    ([i-1]_s,a)\ast([j-1]_s,b)\\
    &=
    \Psi({^{i}a})\ast\Psi({^{j}b}).
\end{align*}

Thus, $\Psi$ is a bijective operation-preserving
map onto an abelian group.

Consequently, the nonzero multisign set is an
abelian group under $\otimes$, and $\Psi$ is
a group isomorphism.
\end{proof}

\begin{theorem}\label{1-abelian_multiplication_proof}
The multiplication $\otimes$ is associative and
commutative on $\sigmac^{s}$, with identity
${^{1}1}$ and absorbing element ${^{0}0}$.

Consequently, $(\sigmac^{s},\otimes)$ is a
commutative monoid, and
\[
    \bigl(
        \sigmac^{s}\setminus\{{^{0}0}\},
        \otimes
    \bigr)
\]
is an abelian group.
\end{theorem}

\begin{proof}
By Proposition~\ref{prop:multiplicative-direct-product},
multiplication is closed, associative, and commutative
on the nonzero multisign set.

The absorbing rule in
Definition~\ref{multiplication} extends these
properties to $\sigmac^{s}$.

Under the isomorphism $\Psi$, the identity
$([0]_s,1)$ corresponds to ${^{1}1}$.
Hence, for every $x\in\sigmac^{s}$,
\[
    x\otimes{^{1}1}
    =
    {^{1}1}\otimes x
    =
    x.
\]

Moreover,
\[
    x\otimes{^{0}0}
    =
    {^{0}0}\otimes x
    =
    {^{0}0}.
\]

Therefore, $(\sigmac^{s},\otimes)$ is a
commutative monoid obtained by adjoining an
absorbing element to an abelian group.
\end{proof}

\begin{corollary}\label{cor:multiplicative-inverse}
Every nonzero multisign number ${^{i}a}$ has
a unique multiplicative inverse given by
\[
    ({^{i}a})^{-1}
    =
    {^{k}(a^{-1})},
\]
where
\[
    k=1+((1-i)\bmod s).
\]

Equivalently,
\[
    ({^{i}a})^{-1}
    =
    \begin{cases}
        {^{1}(a^{-1})}, & i=1,\\
        {^{s+2-i}(a^{-1})}, & i>1.
    \end{cases}
\]
\end{corollary}

\begin{proof}
In $\mathbb{Z}/s\mathbb{Z}\times P^{\times}$,
the inverse of $([i-1]_s,a)$ is
\[
    ([1-i]_s,a^{-1}).
\]

Applying $\Psi^{-1}$ gives the stated formula.
Uniqueness follows from the group structure.
\end{proof}

\begin{remark}
The multiplicative structure is associative for
arbitrary sign combinations. Therefore, it also
satisfies signed-associativity.

Unlike additive inversion, multiplicative
inversion is unique for every nonzero element,
independently of $s$.
\end{remark}

\subsection{Consequences of the multiplicative structure}

\begin{definition}
Let $x\in\sigmac^{s}$.
For every positive integer $m$, we define
\[
    x^{\otimes m}
    =
    \underbrace{x\otimes\cdots\otimes x}
    _{m\text{ factors}}.
\]

For nonzero $x$, we additionally define
\[
    x^{\otimes 0}={^{1}1},
    \qquad
    x^{\otimes(-m)}
    =
    (x^{-1})^{\otimes m}.
\]
\end{definition}

\begin{proposition}\label{prop:multiplicative-powers}
For every nonzero multisign number ${^{i}a}$
and every $m\in\mathbb{Z}$,
\[
    ({^{i}a})^{\otimes m}
    =
    {^{k}(a^m)},
\]
where
\[
    k=1+(m(i-1)\bmod s).
\]
\end{proposition}

\begin{proof}
By Proposition~\ref{prop:multiplicative-direct-product},
\[
    \Psi(( {^{i}a})^{\otimes m})
    =
    ([m(i-1)]_s,a^m).
\]

Applying $\Psi^{-1}$ gives the result.
\end{proof}

\begin{proposition}\label{prop:positive-scaling}
For all $a,b,c\in P$ with $a>0$ and $b<c$,
\[
    ab<ac.
\]
Moreover,
\[
    a(c-b)=ac-ab.
\]
\end{proposition}

\begin{proof}
By Assumption~\ref{ass:P-ordered}, there exists
$d=c-b>0$ such that
\[
    c=b+d.
\]

Since $a,d\neq 0$ and $P^{\times}$ is a group,
$ad\neq 0$.

By distributivity in $P$,
\[
    ac=a(b+d)=ab+ad.
\]

As $ad>0$, compatibility of the order with
addition gives
\[
    ab<ab+ad=ac.
\]

By uniqueness of differences,
\[
    ac-ab=ad=a(c-b).
\]
\end{proof}

\begin{lemma}\label{lem:magnitude-torsion-free}
The multiplicative group $P^{\times}$ is
torsion-free. That is, for every $a\in P^{\times}$
and every positive integer $m$,
\[
    a^m=1\implies a=1.
\]
\end{lemma}

\begin{proof}
If $a>1$, Proposition~\ref{prop:positive-scaling}
gives
\[
    a^{k+1}>a^k
\]
for every nonnegative integer $k$.
Consequently, $a^m>1$.

Similarly, if $0<a<1$, then
\[
    a^{k+1}<a^k,
\]
and consequently $a^m<1$.

Therefore, $a^m=1$ is possible only when $a=1$.
\end{proof}

\begin{proposition}\label{prop:unit-magnitude-group}
The set
\[
    U_s=\{{^{i}1}:i=1,\ldots,s\}
\]
is a cyclic subgroup of
\[
    \bigl(
        \sigmac^{s}\setminus\{{^{0}0}\},
        \otimes
    \bigr),
\]
with order $s$ and generator ${^{2}1}$.

Moreover,
\[
    \operatorname{ord}_{\otimes}({^{i}1})
    =
    \frac{s}{\gcd(s,i-1)}.
\]
\end{proposition}

\begin{proof}
Under the isomorphism $\Psi$, the set $U_s$
corresponds to the subgroup
\[
    \mathbb{Z}/s\mathbb{Z}\times\{1\}.
\]

The element ${^{2}1}$ corresponds to $([1]_s,1)$,
which generates this subgroup.

The order of $([i-1]_s,1)$ is the order of
$[i-1]_s$ in $\mathbb{Z}/s\mathbb{Z}$, namely
\[
    \frac{s}{\gcd(s,i-1)}.
\]

The result follows from preservation of element
orders under group isomorphisms.
\end{proof}

\begin{theorem}\label{thm:multiplicative-torsion}
A nonzero multisign number has finite
multiplicative order if and only if its
magnitude equals $1$.

Consequently, $U_s$ is the torsion subgroup
of the nonzero multiplicative group.
\end{theorem}

\begin{proof}
Suppose that ${^{i}a}$ has finite multiplicative
order $m$.

Then
\[
    ({^{i}a})^{\otimes m}={^{1}1}.
\]

By Proposition~\ref{prop:multiplicative-powers},
this implies
\[
    a^m=1.
\]

By Lemma~\ref{lem:magnitude-torsion-free},
$a=1$.

Conversely, if $a=1$, then
Proposition~\ref{prop:unit-magnitude-group}
establishes that ${^{i}1}$ has finite
multiplicative order.

Hence, the elements of finite multiplicative
order are exactly the elements of $U_s$.
\end{proof}

\begin{corollary}
For every nonzero ${^{i}a}\in\sigmac^{s}$,
\[
    ({^{i}a})^{\otimes s}
    =
    {^{1}(a^s)}.
\]
\end{corollary}

\begin{proof}
By Proposition~\ref{prop:multiplicative-powers},
\[
    s(i-1)\equiv 0\pmod{s}.
\]
\end{proof}

\subsection{Multiplicative equations}

\begin{proposition}\label{prop:multiplicative-equations}
Let $u,v\in\sigmac^{s}$.

If $u\neq{^{0}0}$, the equation
\[
    x\otimes u=v
\]
has the unique solution
\[
    x=v\otimes u^{-1}.
\]

If $u={^{0}0}$, the equation has every element
of $\sigmac^{s}$ as a solution when
$v={^{0}0}$, and no solution otherwise.
\end{proposition}

\begin{proof}
If $u\neq{^{0}0}$, associativity and the
existence of $u^{-1}$ give
\begin{align*}
    (x\otimes u)\otimes u^{-1}
    &=
    x\otimes(u\otimes u^{-1})\\
    &=
    x\otimes{^{1}1}\\
    &=
    x.
\end{align*}

Thus,
\[
    x=v\otimes u^{-1}.
\]

Conversely,
\[
    (v\otimes u^{-1})\otimes u=v,
\]
which establishes existence and uniqueness.

The case $u={^{0}0}$ follows from the
absorbing property.
\end{proof}

\begin{theorem}\label{thm:multiplicative-roots}
Let $m$ be a positive integer and let
${^{i}a}\in\sigmac^{s}$ be nonzero.

Set
\[
    g=\gcd(m,s).
\]

The equation
\[
    x^{\otimes m}={^{i}a}
\]
has a solution if and only if both of the
following conditions hold:

\begin{enumerate}[label=\textnormal{(\roman*)}]

    \item There exists $b\in P^{\times}$ such that
    \[
        b^m=a.
    \]

    \item The integer $i-1$ is divisible by $g$.

\end{enumerate}

If these conditions hold, the equation has
exactly $g$ distinct solutions.
Otherwise, it has no solution.
\end{theorem}

\begin{proof}
Since the right-hand side is nonzero,
every solution has the form ${^{j}b}$ with
$b\neq 0$.

By Proposition~\ref{prop:multiplicative-powers},
\[
    ({^{j}b})^{\otimes m}={^{i}a}
\]
if and only if
\[
    b^m=a
\]
and
\[
    m(j-1)\equiv i-1\pmod{s}.
\]

Suppose that $b_1^m=b_2^m=a$.
Since $P^{\times}$ is an abelian group,
\[
    (b_1b_2^{-1})^m=1.
\]

By Lemma~\ref{lem:magnitude-torsion-free},
\[
    b_1b_2^{-1}=1,
\]
so $b_1=b_2$.

Hence, the magnitude equation admits at most
one solution.

The linear congruence
\[
    m(j-1)\equiv i-1\pmod{s}
\]
has a solution only if $g\mid(i-1)$.
If this divisibility holds, division by $g$ yields a
congruence modulo $s/g$ with coefficient $m/g$,
which is coprime to $s/g$. It has one residue class
of solutions modulo $s/g$, and hence exactly $g$
distinct residue classes modulo $s$.

Therefore, if the magnitude equation is
solvable and $g\mid(i-1)$, there are exactly
$g$ distinct multisign solutions.
\end{proof}

\begin{corollary}\label{cor:roots-of-unity}
For every positive integer $m$, the equation
\[
    x^{\otimes m}={^{1}1}
\]
has exactly $\gcd(m,s)$ distinct solutions,
all of magnitude $1$.
\end{corollary}

\begin{proof}
The magnitude equation $b^m=1$ has the
unique solution $b=1$.

Since $\gcd(m,s)$ divides $0$, the result
follows from
Theorem~\ref{thm:multiplicative-roots}.
\end{proof}

For example, consider $s=4$ and
$P=\mathbb{R}_{\geq 0}$.

The equation
\[
    x^{\otimes 2}={^{1}4}
\]
has two solutions,
\[
    x\in\{{^{1}2},{^{3}2}\}.
\]

Similarly,
\[
    x^{\otimes 2}={^{3}4}
\]
has the two solutions
\[
    x\in\{{^{2}2},{^{4}2}\}.
\]

In contrast,
\[
    x^{\otimes 2}={^{2}4}
\]
has no solution, since $\gcd(2,4)=2$
does not divide $2-1$.

\subsection{Equations involving addition and multiplication}
\label{subsec:example-mixed}

\begin{proposition}\label{prop:mixed-equations}
Let $u,v,w\in\sigmac^{s}$, with $u\neq{^{0}0}$.
The solution set of
\[
    (x\otimes u)\oplus v=w
\]
is
\[
    \left\{
        y\otimes u^{-1}:y\in E(v,w)
    \right\},
\]
where
\[
    E(v,w)=\{y\in\sigmac^{s}:y\oplus v=w\}.
\]

Consequently, the equation has either one
or $s-1$ distinct solutions.
\end{proposition}

\begin{proof}
Set
\[
    y=x\otimes u.
\]

The equation becomes
\[
    y\oplus v=w.
\]

By Proposition~\ref{prop:additive-equations},
the solution set of this equation is $E(v,w)$.

By Proposition~\ref{prop:multiplicative-equations},
each $y\in E(v,w)$ determines a unique solution
\[
    x=y\otimes u^{-1}.
\]

Multiplication by $u^{-1}$ is bijective on
$\sigmac^{s}$ and therefore preserves the
cardinality of $E(v,w)$.

By Corollary~\ref{cor:additive-equation-cardinality},
\[
    |E(v,w)|\in\{1,s-1\}.
\]

The result follows.
\end{proof}

\begin{corollary}
For every nonzero $u\in\sigmac^{s}$, the equation
\[
    (x\otimes u)\oplus u={^{0}0}
\]
has exactly $s-1$ distinct solutions, namely
\[
    x\in U_s\setminus\{{^{1}1}\}.
\]
\end{corollary}

\begin{proof}
Write $u={^{i}a}$ and set $y=x\otimes u$.

By Proposition~\ref{prop:additive-inverses},
the equation
\[
    y\oplus u={^{0}0}
\]
holds if and only if
\[
    y={^{j}a},
    \qquad
    j\in D\setminus\{0,i\}.
\]

Under the isomorphism $\Psi$,
\[
    \Psi(y\otimes u^{-1})
    =
    ([j-i]_s,1).
\]

As $j$ ranges over the nonzero sign labels
different from $i$, $[j-i]_s$ ranges over
all nonzero residue classes modulo $s$.

Therefore,
\[
    x\in U_s\setminus\{{^{1}1}\}.
\]
\end{proof}

For example, when $s=4$ and
$P=\mathbb{R}_{\geq 0}$, the equation
\[
    (x\otimes{^{1}2})\oplus{^{1}2}
    =
    {^{0}0}
\]
has exactly three solutions:
\[
    x\in\{{^{2}1},{^{3}1},{^{4}1}\}.
\]

\begin{remark}\label{classical-multiplication-s2}
For $s=2$, the nonzero multiplicative group
is isomorphic to
\[
    \mathbb{Z}/2\mathbb{Z}\times P^{\times}.
\]

When $P$ is the nonnegative cone of an
ordered field $K$, this group is isomorphic
to the multiplicative group $K^{\times}$.

Section~\ref{real_numbers} relates the two-sign
construction to classical signed arithmetic.
\end{remark}

\subsection[ni-Signed-Ring]{\texorpdfstring{$n$i-Signed-Ring}{ni-Signed-Ring}}

\begin{definition}\label{def-s-ring}
An \textbf{$n$i-signed-ring}
$(R,\oplus,\otimes)$ is a sign-labelled set with two
binary operations satisfying the following conditions
with respect to the same fixed sign labelling:

\begin{enumerate}[label=\textnormal{\arabic*.}]

    \item $(R,\oplus)$ is an abelian
    $n$i-signed-group with identity $0_R$.

    \item $(R,\otimes)$ is a
    signed-monoid with identity $1_R$.

    \item Multiplication distributes over addition
    on both sides. For all $x,y,z\in R$,
    \begin{align*}
        x\otimes(y\oplus z)
        &=
        (x\otimes y)\oplus(x\otimes z),\\
        (x\oplus y)\otimes z
        &=
        (x\otimes z)\oplus(y\otimes z).
    \end{align*}

    \item The additive identity is absorbing under
    multiplication. For all $x\in R$,
    \[
        x\otimes 0_R
        =
        0_R\otimes x
        =
        0_R.
    \]

\end{enumerate}

We call an $n$i-signed-ring
\textbf{commutative} if
\[
    x\otimes y=y\otimes x
\]
for all $x,y\in R$.
\end{definition}

\begin{remark}
The parameter $n$ refers to the number of
bilateral additive inverses of each element
different from $0_R$.

Unlike classical rings, $n$i-signed-rings
require only signed-associativity of addition
and multiplication.

We include the absorbing property of $0_R$
as an explicit axiom because additive
cancellation is not assumed.
\end{remark}

\begin{proposition}\label{prop:multisign-distributivity}
The multiplication defined in
Definition~\ref{multiplication} distributes
over the addition defined in
Definition~\ref{addition}.
\end{proposition}

\begin{proof}
Let $x={^{i}a}$, $y={^{j}b}$, and
$z={^{k}c}$ be elements of $\sigmac^{s}$.

If any operand equals ${^{0}0}$, the left
distributive law follows from the additive
identity and absorbing properties.

Suppose that all operands are nonzero.

For a fixed sign $i$, define
\[
    \pi_i(j)=1+((i+j-2)\bmod s).
\]

The map $\pi_i$ is a permutation of
$D\setminus\{0\}$. Hence,
\[
    j=k\iff\pi_i(j)=\pi_i(k).
\]

By Proposition~\ref{prop:positive-scaling},
multiplication by $a>0$ preserves the order
of magnitudes and their differences.

If $j=k$, distributivity in $P$ gives
\begin{align*}
    x\otimes(y\oplus z)
    &=
    {^{\pi_i(j)}(a(b+c))}\\
    &=
    {^{\pi_i(j)}(ab+ac)}\\
    &=
    (x\otimes y)\oplus(x\otimes z).
\end{align*}

If $j\neq k$ and $b>c$, then $ab>ac$.
Therefore,
\begin{align*}
    x\otimes(y\oplus z)
    &=
    {^{\pi_i(j)}(a(b-c))}\\
    &=
    {^{\pi_i(j)}(ab-ac)}\\
    &=
    (x\otimes y)\oplus(x\otimes z).
\end{align*}

The case $j\neq k$ and $c>b$ follows
analogously.

If $j\neq k$ and $b=c$, then
$\pi_i(j)\neq\pi_i(k)$ and $ab=ac$.
Consequently,
\[
    x\otimes(y\oplus z)
    =
    {^{0}0}
    =
    (x\otimes y)\oplus(x\otimes z).
\]

Thus, the left distributive law holds.

The right distributive law follows from
the commutativity of $\otimes$ established
in Theorem~\ref{1-abelian_multiplication_proof}.
\end{proof}

\begin{theorem}\label{s-ring}
The structure $(\sigmac^{s},\oplus,\otimes)$
is a commutative $(s-1)$i-signed-ring with
additive identity ${^{0}0}$ and multiplicative
identity ${^{1}1}$.
\end{theorem}

\begin{proof}
By Theorem~\ref{s-abelian_group_proof},
$(\sigmac^{s},\oplus)$ is an abelian
$(s-1)$i-signed-group.

By Theorem~\ref{1-abelian_multiplication_proof},
$(\sigmac^{s},\otimes)$ is a commutative
monoid with identity ${^{1}1}$ and absorbing
element ${^{0}0}$.

By Proposition~\ref{prop:multisign-distributivity},
multiplication distributes over addition
on both sides.

Hence, all the conditions of
Definition~\ref{def-s-ring} hold.
\end{proof}

\begin{remark}
A classical unital ring equipped with a sign labelling
is a $1$i-signed-ring.

Conversely, an $n$i-signed-ring is a
classical unital ring if and only if both operations
are fully associative. In this case, $n=1$ whenever
$R$ has more than one element.

For the multisign construction, multiplication
is fully associative for every $s\geq 2$,
whereas addition is fully associative
if and only if $s=2$.
\end{remark}

\begin{proposition}\label{prop:multiplicative-additive-automorphism}
For every nonzero $u\in\sigmac^{s}$, the map
\[
    M_u:\sigmac^{s}\longrightarrow\sigmac^{s},
    \qquad
    M_u(x)=u\otimes x,
\]
is an automorphism of the additive structure
$(\sigmac^{s},\oplus)$.
\end{proposition}

\begin{proof}
By Proposition~\ref{prop:multisign-distributivity},
\[
    M_u(x\oplus y)
    =
    M_u(x)\oplus M_u(y).
\]

Moreover,
\[
    M_u({^{0}0})={^{0}0}.
\]

Since $u$ is multiplicatively invertible,
full associativity gives
\[
    M_{u^{-1}}\circ M_u
    =
    M_u\circ M_{u^{-1}}
    =
    \operatorname{id}_{\sigmac^{s}}.
\]

Therefore, $M_u$ is a bijective
addition-preserving map.
\end{proof}

\begin{corollary}\label{cor:scaled-additive-inverses}
For all $x\in\sigmac^{s}$ and every nonzero
$u\in\sigmac^{s}$,
\[
    I_{\oplus}(u\otimes x)
    =
    \{
        u\otimes y:
        y\in I_{\oplus}(x)
    \}.
\]
\end{corollary}

\begin{proof}
If $y\in I_{\oplus}(x)$, then
\begin{align*}
    (u\otimes x)\oplus(u\otimes y)
    &=
    u\otimes(x\oplus y)\\
    &=
    u\otimes{^{0}0}\\
    &=
    {^{0}0}.
\end{align*}

Thus, $u\otimes y\in I_{\oplus}(u\otimes x)$.

Conversely, applying the inverse automorphism
$M_{u^{-1}}$ to any element of
$I_{\oplus}(u\otimes x)$ gives an element
of $I_{\oplus}(x)$.

The equality follows.
\end{proof}

\begin{corollary}\label{cor:regular-multiplicative-action}
The map
\[
    u\longmapsto M_u
\]
defines an injective group homomorphism
\[
    \bigl(
        \sigmac^{s}\setminus\{{^{0}0}\},
        \otimes
    \bigr)
    \longrightarrow
    \operatorname{Aut}(\sigmac^{s},\oplus).
\]

Its image acts freely and transitively on
$\sigmac^{s}\setminus\{{^{0}0}\}$.
\end{corollary}

\begin{proof}
For nonzero $u,v$ and every $x\in\sigmac^{s}$,
\[
    M_{u\otimes v}(x)
    =
    M_u(M_v(x)).
\]

Thus,
\[
    M_{u\otimes v}=M_u\circ M_v.
\]

If $M_u=M_v$, evaluating both maps at
${^{1}1}$ gives $u=v$. Hence, the
homomorphism is injective.

For any nonzero $x,y\in\sigmac^{s}$,
there exists exactly one nonzero $u$
satisfying
\[
    M_u(x)=y,
\]
namely
\[
    u=y\otimes x^{-1}.
\]

Therefore, the action is free and transitive.
\end{proof}

\begin{remark}\label{classical-integers-s2}
The additive multisign construction requires
only a nontrivial totally ordered cancellative commutative
monoid of magnitudes with least element $0$
satisfying
conditions (i)--(iv) of Assumption~\ref{ass:P-ordered}.

Multiplicative invertibility in $P$ is not
required for the additive results.

For example, $\mathbb{Z}_{\geq 0}$ satisfies
these additive conditions.

Applying Proposition~\ref{prop:two-sign-group}
with $P=\mathbb{Z}_{\geq 0}$ gives
\[
    (\sigmac_{\mathbb{Z}_{\geq 0}}^{2},\oplus)
    \cong
    (\mathbb{Z},+).
\]

This example uses only the additive assumptions.
\end{remark}

\subsection[ni-Signed-Field]{\texorpdfstring{$n$i-Signed-Field}{ni-Signed-Field}}

\begin{definition}\label{def:ni-signed-field}
An \textbf{$n$i-signed-field}
$(F,\oplus,\otimes)$ is a commutative
$n$i-signed-ring satisfying the following
additional conditions:

\begin{enumerate}[label=\textnormal{\arabic*.}]

    \item The additive and multiplicative
    identities are distinct:
    \[
        0_F\neq 1_F.
    \]

    \item The structure
    \[
        (F\setminus\{0_F\},\otimes)
    \]
    is an abelian group with identity $1_F$.

\end{enumerate}
\end{definition}

\begin{remark}
An $n$i-signed-field retains the multiplicative
group structure, the absorbing property of zero,
and the distributive laws of a classical field.

Its additive structure, however, requires
signed-associativity rather than full
associativity and admits exactly $n$ bilateral
inverses for every nonzero element.

For $n>1$, it cannot be a classical field.
For $n=1$, it is a classical field if and only
if addition is fully associative.
\end{remark}

\begin{theorem}\label{s-field}
The structure
\[
    F=(\sigmac^{s},\oplus,\otimes)
\]
is a commutative $(s-1)$i-signed-field.
\end{theorem}

\begin{proof}
By Theorem~\ref{s-ring}, $F$ is a
commutative $(s-1)$i-signed-ring.

Its additive and multiplicative identities
satisfy
\[
    {^{0}0}\neq{^{1}1}.
\]

By Proposition~\ref{prop:multiplicative-direct-product},
the nonzero multiplicative structure is an
abelian group.

Therefore, $F$ satisfies
Definition~\ref{def:ni-signed-field}.
\end{proof}

\begin{proposition}\label{prop:no-zero-divisors}
An $n$i-signed-field has no nonzero
zero divisors.

For all $x,y\in F$,
\[
    x\otimes y=0_F
    \iff
    x=0_F\ \lor\ y=0_F.
\]
\end{proposition}

\begin{proof}
If $x=0_F$ or $y=0_F$, the absorbing
property gives
\[
    x\otimes y=0_F.
\]

Conversely, if $x\neq 0_F$ and $y\neq 0_F$,
closure of the multiplicative group
$F\setminus\{0_F\}$ implies
\[
    x\otimes y\neq 0_F.
\]
\end{proof}

\begin{proposition}\label{prop:multiplicative-cancellation-field}
Every $n$i-signed-field satisfies
multiplicative cancellation for nonzero
factors.

For all $x,y,u\in F$ with $u\neq 0_F$,
\[
    x\otimes u=y\otimes u
    \implies
    x=y.
\]
\end{proposition}

\begin{proof}
If either $x$ or $y$ equals $0_F$,
Proposition~\ref{prop:no-zero-divisors}
implies that both equal $0_F$.

Otherwise, $x$, $y$, and $u$ belong to
the multiplicative group of nonzero elements.

Multiplying both sides by $u^{-1}$ gives
$x=y$.
\end{proof}

\begin{proposition}\label{prop:multiplicative-inverse-action}
Let $F$ be an $n$i-signed-field.
For every nonzero $x\in F$,
\[
    I_{\oplus}(x)
    =
    \{
        x\otimes u:
        u\in I_{\oplus}(1_F)
    \}.
\]
\end{proposition}

\begin{proof}
For every $u\in I_{\oplus}(1_F)$,
distributivity gives
\begin{align*}
    x\oplus(x\otimes u)
    &=
    (x\otimes 1_F)\oplus(x\otimes u)\\
    &=
    x\otimes(1_F\oplus u)\\
    &=
    x\otimes 0_F\\
    &=
    0_F.
\end{align*}

By commutativity of addition,
$x\otimes u$ is a bilateral additive
inverse of $x$.

Conversely, let $y\in I_{\oplus}(x)$.
Then $y\neq 0_F$, and
\begin{align*}
    1_F\oplus(x^{-1}\otimes y)
    &=
    (x^{-1}\otimes x)
    \oplus(x^{-1}\otimes y)\\
    &=
    x^{-1}\otimes(x\oplus y)\\
    &=
    0_F.
\end{align*}

Thus,
\[
    x^{-1}\otimes y\in I_{\oplus}(1_F).
\]

Setting $u=x^{-1}\otimes y$ gives
$y=x\otimes u$, which establishes the
equality.
\end{proof}

\begin{corollary}
In the multisign construction,
\[
    I_{\oplus}({^{1}1})
    =
    U_s\setminus\{{^{1}1}\},
\]
where $U_s$ is the cyclic subgroup defined
in Proposition~\ref{prop:unit-magnitude-group}.

Consequently, for every nonzero
$x\in\sigmac^{s}$,
\[
    I_{\oplus}(x)
    =
    \{
        x\otimes u:
        u\in U_s\setminus\{{^{1}1}\}
    \}.
\]
\end{corollary}

\begin{proof}
By Proposition~\ref{prop:additive-inverses},
the additive inverses of ${^{1}1}$ are
\[
    {^{2}1},\ldots,{^{s}1}.
\]

These are exactly the elements of
$U_s\setminus\{{^{1}1}\}$.

The result follows from
Proposition~\ref{prop:multiplicative-inverse-action}.
\end{proof}

\begin{proposition}\label{prop:unit-orbits}
For every nonzero $x\in\sigmac^{s}$,
\[
    \{x\}\cup I_{\oplus}(x)
    =
    \{x\otimes u:u\in U_s\}.
\]

Moreover, for nonzero $x,y\in\sigmac^{s}$,
\[
    |x|=|y|
    \iff
    y\in\{x\otimes u:u\in U_s\}.
\]
\end{proposition}

\begin{proof}
The first equality follows from
Proposition~\ref{prop:multiplicative-inverse-action}
and the characterization of
$I_{\oplus}({^{1}1})$.

Write $x={^{i}a}$.
Multiplication by an element of $U_s$
preserves its magnitude and permutes
its sign.

As $u$ ranges over $U_s$, the sign of
$x\otimes u$ ranges over all nonzero
sign labels.

Therefore,
\[
    \{x\otimes u:u\in U_s\}
    =
    \{{^{j}a}:j=1,\ldots,s\}.
\]

The second equivalence follows.
\end{proof}

\begin{corollary}\label{cor:magnitude-quotient}
The magnitude map
\[
    |\cdot|:
    \sigmac^{s}\setminus\{{^{0}0}\}
    \longrightarrow P^{\times}
\]
is a surjective group homomorphism with
kernel $U_s$.

Consequently,
\[
    \frac{
        \sigmac^{s}\setminus\{{^{0}0}\}
    }{U_s}
    \cong P^{\times}
\]
as multiplicative groups.
\end{corollary}

\begin{proof}
For all nonzero $x,y$,
\[
    |x\otimes y|=|x||y|.
\]

Every $a\in P^{\times}$ is the magnitude
of ${^{1}a}$, so the map is surjective.

Its kernel consists of the elements of
magnitude $1$, namely $U_s$.

The result follows from the first
isomorphism theorem for groups.
\end{proof}

\begin{theorem}\label{thm:classical-subfield}
Let $F=(\sigmac^{s},\oplus,\otimes)$.
By a classical subfield of $F$, we mean a subset
that is a classical field under the restricted operations,
with multiplicative identity ${^{1}1}$.
Then $F$ contains a classical subfield
if and only if $s$ is even.

When $s$ is even, the set
\[
    C_s
    =
    \{{^{0}0}\}
    \cup
    \{
        {^{1}a}:a\in P^{\times}
    \}
    \cup
    \{
        {^{1+s/2}a}:a\in P^{\times}
    \}
\]
is the greatest classical subfield of $F$,
under inclusion.
\end{theorem}

\begin{proof}
Suppose that $K\subseteq \sigmac^{s}$ is a classical
subfield containing ${^{1}1}$. Its additive identity is
idempotent under $\oplus$. If $x={^{i}a}$ is nonzero, then
$x\oplus x={^{i}(a+a)}\neq x$ by additive cancellation
in $P$. Thus, the additive identity of $K$ is ${^{0}0}$.

Let $u$ be the additive inverse of
${^{1}1}$ in $K$.
By Proposition~\ref{prop:additive-inverses},
\[
    u={^{j}1},
    \qquad j\neq 1.
\]

Since $K$ is a classical field,
\[
    u\otimes u={^{1}1}.
\]

Since $j\neq 1$, we have $u\neq {^{1}1}$.
Thus, $u$ has multiplicative order $2$.

By Proposition~\ref{prop:unit-magnitude-group},
this requires
\[
    \frac{s}{\gcd(s,j-1)}=2.
\]

Hence, $s$ must be even, and the only
possible sign of $u$ is
\[
    j=1+\frac{s}{2}.
\]

Conversely, suppose that $s$ is even.
Set
\[
    q=1+\frac{s}{2}.
\]

By Proposition~\ref{prop:two-sign-group},
\[
    (C_s,\oplus)=(S_{1q},\oplus)
\]
is an abelian group.

The residues $[0]_s$ and $[s/2]_s$ form a
subgroup of $\mathbb{Z}/s\mathbb{Z}$.
Its direct product with $P^{\times}$ is the image
of $C_s\setminus\{{^{0}0}\}$ under $\Psi$.
Therefore, by
Proposition~\ref{prop:multiplicative-direct-product},
the nonzero part of $C_s$ is an abelian
group under multiplication.

Both distributive laws hold on $C_s$ because
they hold on $\sigmac^{s}$.

Consequently, $C_s$ satisfies the axioms
of a classical field.

Finally, every classical subfield with
multiplicative identity ${^{1}1}$ contains
both ${^{1}1}$ and ${^{q}1}$.

By Proposition~\ref{prop:maximal-associative-subsets},
an associative additive substructure
cannot contain three distinct nonzero
signs.

Hence, every such classical subfield is
contained in $C_s$.

Therefore, $C_s$ is the greatest classical
subfield of $F$ under inclusion.
\end{proof}

\begin{remark}
For $s=2$, we have $C_s=\sigmac^{2}$.

For even $s>2$, $C_s$ is a proper
classical subfield of $\sigmac^{s}$.

For odd $s$, the multisign construction
contains subsets that are abelian groups under
restricted addition, but no classical subfield
sharing its multiplicative identity.
\end{remark}

\subsection[ni-Signed-Module]{\texorpdfstring{$n$i-Signed-Module}{ni-Signed-Module}}

\begin{definition}\label{def:ni-signed-module}
Let $(R,\oplus_R,\otimes_R)$ be a commutative
$n$i-signed-ring with additive identity $0_R$
and multiplicative identity $1_R$.

An \textbf{$R$-$n$i-signed-module} is an abelian
$n$i-signed-group $(L,\oplus_L)$ on a sign-labelled
set $L$ with additive identity $0_L$, equipped with a scalar action
\[
    \cdot:R\times L\longrightarrow L
\]
satisfying the following conditions for all
$r,t\in R$ and $x,y\in L$:

\begin{enumerate}[label=\textnormal{(\roman*)}]

    \item Distributivity over module addition:
    \[
        r\cdot(x\oplus_L y)
        =
        (r\cdot x)\oplus_L(r\cdot y).
    \]

    \item Distributivity over scalar addition:
    \[
        (r\oplus_R t)\cdot x
        =
        (r\cdot x)\oplus_L(t\cdot x).
    \]

    \item Compatibility with scalar multiplication:
    \[
        (r\otimes_R t)\cdot x
        =
        r\cdot(t\cdot x).
    \]

    \item Unitary action:
    \[
        1_R\cdot x=x.
    \]

    \item Zero scalar action:
    \[
        0_R\cdot x=0_L.
    \]

    \item Preservation of the additive identity:
    \[
        r\cdot0_L=0_L.
    \]

\end{enumerate}
\end{definition}

\begin{remark}
The scalar ring $R$ and the additive structure
$L$ share the parameter $n$. We specify their sign labellings independently and
assume no compatibility beyond the stated axioms.

We require the scalar-action identities for all
operands, regardless of their sign labels.

We state the two zero-action identities explicitly
because the axioms do not require additive cancellation.
\end{remark}

\begin{proposition}\label{prop:regular-signed-module}
Every commutative $n$i-signed-ring with fully
associative multiplication is an
$R$-$n$i-signed-module over itself under the
scalar action
\[
    r\cdot x=r\otimes_R x.
\]
\end{proposition}

\begin{proof}
By Definition~\ref{def-s-ring},
$(R,\oplus_R)$ is an abelian
$n$i-signed-group.

Both scalar distributive identities follow
from the distributive laws in $R$.

Full multiplicative associativity gives
\[
    (r\otimes_R t)\otimes_R x
    =
    r\otimes_R(t\otimes_R x).
\]

The multiplicative identity gives
\[
    1_R\otimes_R x=x.
\]

The absorbing property of $0_R$ gives
\[
    0_R\otimes_R x
    =
    r\otimes_R0_R
    =
    0_R.
\]

Thus, all the conditions of
Definition~\ref{def:ni-signed-module} hold.
\end{proof}

\begin{corollary}\label{cor:multisign-module}
The structure
\[
    L=(\sigmac^{s},\oplus)
\]
is an $R$-$(s-1)$i-signed-module over
\[
    R=(\sigmac^{s},\oplus,\otimes),
\]
with scalar action
\[
    r\cdot x=r\otimes x.
\]
\end{corollary}

\begin{proof}
By Theorem~\ref{s-ring}, $R$ is a
commutative $(s-1)$i-signed-ring.

By Theorem~\ref{1-abelian_multiplication_proof},
its multiplication is fully associative.

The result follows from
Proposition~\ref{prop:regular-signed-module}.
\end{proof}

\begin{proposition}\label{prop:scalar-action-inverses}
Let $L$ be an $R$-$n$i-signed-module.
For all $r\in R$ and $x\in L$,
\[
    \{r\cdot y:y\in I_{\oplus_L}(x)\}
    \subseteq
    I_{\oplus_L}(r\cdot x).
\]
\end{proposition}

\begin{proof}
Let $y\in I_{\oplus_L}(x)$.
Then
\[
    x\oplus_L y=0_L.
\]

By distributivity and the zero-action axioms,
\[
    (r\cdot x)\oplus_L(r\cdot y)
    =
    r\cdot(x\oplus_L y)
    =
    r\cdot0_L
    =
    0_L.
\]

Since addition in $L$ is commutative,
$r\cdot y$ is a bilateral inverse
of $r\cdot x$.
\end{proof}

\subsection[ni-Signed-Vector Space]{\texorpdfstring{$n$i-Signed-Vector Space}{ni-Signed-Vector Space}}

\begin{definition}\label{def:ni-signed-vector-space}
Let $F$ be an $n$i-signed-field.
An \textbf{$n$i-signed-vector space} over $F$
is an $F$-$n$i-signed-module in the sense of
Definition~\ref{def:ni-signed-module}.
\end{definition}

\begin{remark}
An $n$i-signed-vector space satisfies the
distributive, associative, unitary, and
zero-action identities of scalar multiplication.

Its additive structure satisfies
signed-associativity and admits exactly $n$
bilateral additive inverses for each nonzero
vector.

For $n>1$, it is not a classical vector space over $F$,
since $F$ is not a classical field.
For $n=1$, it is a classical vector space over $F$ if
and only if scalar addition and vector addition are fully
associative.
\end{remark}

\begin{proposition}\label{prop:scalar-automorphism}
Let $V$ be an $n$i-signed-vector space over $F$.
For every $\lambda\in F\setminus\{0_F\}$,
the map
\[
    T_{\lambda}:V\longrightarrow V,
    \qquad
    T_{\lambda}(v)=\lambda\cdot v,
\]
is an automorphism of $(V,\oplus_V)$, with
inverse $T_{\lambda^{-1}}$.
\end{proposition}

\begin{proof}
By distributivity,
\[
    T_{\lambda}(u\oplus_V v)
    =
    T_{\lambda}(u)\oplus_V T_{\lambda}(v).
\]

The zero-action axiom gives
\[
    T_{\lambda}(0_V)=0_V.
\]

Since $\lambda\neq 0_F$, its multiplicative
inverse exists. By scalar compatibility and
the unitary action,
\begin{align*}
    T_{\lambda^{-1}}(T_{\lambda}(v))
    &=
    \lambda^{-1}\cdot(\lambda\cdot v)\\
    &=
    (\lambda^{-1}\otimes_F\lambda)\cdot v\\
    &=
    1_F\cdot v\\
    &=v.
\end{align*}

Similarly,
\[
    T_{\lambda}(T_{\lambda^{-1}}(v))=v.
\]

Thus, $T_{\lambda}$ is a bijective
addition-preserving map.
\end{proof}

\begin{corollary}\label{cor:scalar-zero-product}
Let $V$ be an $n$i-signed-vector space over $F$.
For all $\lambda\in F$ and $v\in V$,
\[
    \lambda\cdot v=0_V
    \iff
    \lambda=0_F\ \lor\ v=0_V.
\]
\end{corollary}

\begin{proof}
If $\lambda=0_F$ or $v=0_V$, the result
follows from the zero-action axioms.

Conversely, suppose that
$\lambda\cdot v=0_V$ and $\lambda\neq 0_F$.
Then
\begin{align*}
    v
    &=
    \lambda^{-1}\cdot(\lambda\cdot v)\\
    &=
    \lambda^{-1}\cdot0_V\\
    &=
    0_V.
\end{align*}
\end{proof}

\begin{corollary}\label{cor:scalar-cancellation}
Let $V$ be an $n$i-signed-vector space over $F$.
For every $\lambda\neq 0_F$ and all $u,v\in V$,
\[
    \lambda\cdot u=\lambda\cdot v
    \implies u=v.
\]
\end{corollary}

\begin{proof}
The map $T_{\lambda}$ is injective by
Proposition~\ref{prop:scalar-automorphism}.
\end{proof}

This result concerns cancellation of a fixed nonzero scalar.
It does not establish injectivity of the map
$\lambda\mapsto\lambda\cdot v$ for a fixed nonzero vector $v$.

\begin{corollary}\label{cor:vector-inverse-equivariance}
Let $V$ be an $n$i-signed-vector space over $F$.
For every $v\in V$ and $\lambda\neq 0_F$,
\[
    I_{\oplus_V}(\lambda\cdot v)
    =
    \{\lambda\cdot u:u\in I_{\oplus_V}(v)\}.
\]
\end{corollary}

\begin{proof}
By Proposition~\ref{prop:scalar-action-inverses},
\[
    \{\lambda\cdot u:u\in I_{\oplus_V}(v)\}
    \subseteq I_{\oplus_V}(\lambda\cdot v).
\]

Conversely, let
$w\in I_{\oplus_V}(\lambda\cdot v)$.
Applying $T_{\lambda^{-1}}$ gives
\[
    v\oplus_V(\lambda^{-1}\cdot w)=0_V.
\]

Hence,
\[
    \lambda^{-1}\cdot w\in I_{\oplus_V}(v).
\]

Since
\[
    w=\lambda\cdot(\lambda^{-1}\cdot w),
\]
the reverse inclusion follows.
\end{proof}

\begin{proposition}\label{prop:scalar-inverse-generation}
Let $V$ be an $n$i-signed-vector space over $F$.
For every $v\in V$,
\[
    \{\lambda\cdot v:
    \lambda\in I_{\oplus_F}(1_F)\}
    \subseteq I_{\oplus_V}(v).
\]

If $v\neq 0_V$ and the map
\[
    \lambda\longmapsto\lambda\cdot v
\]
is injective on $I_{\oplus_F}(1_F)$, then
\[
    I_{\oplus_V}(v)
    =
    \{\lambda\cdot v:
    \lambda\in I_{\oplus_F}(1_F)\}.
\]
\end{proposition}

\begin{proof}
For every $\lambda\in I_{\oplus_F}(1_F)$,
scalar distributivity gives
\begin{align*}
    v\oplus_V(\lambda\cdot v)
    &=
    (1_F\cdot v)\oplus_V(\lambda\cdot v)\\
    &=
    (1_F\oplus_F\lambda)\cdot v\\
    &=
    0_F\cdot v\\
    &=
    0_V.
\end{align*}

By commutativity of vector addition,
$\lambda\cdot v$ is a bilateral additive
inverse of $v$. Thus,
\[
    \{\lambda\cdot v:
    \lambda\in I_{\oplus_F}(1_F)\}
    \subseteq I_{\oplus_V}(v).
\]

If $v\neq 0_V$, both $I_{\oplus_F}(1_F)$
and $I_{\oplus_V}(v)$ contain exactly $n$
elements.

Under the injectivity assumption, the set
\[
    \{\lambda\cdot v:
    \lambda\in I_{\oplus_F}(1_F)\}
\]
also contains exactly $n$ elements.
Hence, the inclusion is an equality.
\end{proof}

\begin{theorem}\label{s-vector_space}
The structure
\[
    V=(\sigmac^{s},\oplus)
\]
is an $(s-1)$i-signed-vector space over
\[
    F=(\sigmac^{s},\oplus,\otimes),
\]
with scalar action
\[
    \lambda\cdot v=\lambda\otimes v.
\]
\end{theorem}

\begin{proof}
By Theorem~\ref{s-field}, $F$ is an
$(s-1)$i-signed-field.

By Corollary~\ref{cor:multisign-module},
$(\sigmac^{s},\oplus)$ is an
$(s-1)$i-signed-module over $F$.

The result follows from
Definition~\ref{def:ni-signed-vector-space}.
\end{proof}

\begin{corollary}\label{cor:multisign-vector-inverses}
In the signed-vector space of
Theorem~\ref{s-vector_space}, every nonzero
$v\in V$ satisfies
\[
    I_{\oplus_V}(v)
    =
    \{\lambda\cdot v:
    \lambda\in I_{\oplus_F}({^{1}1})\}.
\]

Equivalently,
\[
    I_{\oplus_V}(v)
    =
    \{{^{j}1}\cdot v:j=2,\ldots,s\}.
\]
\end{corollary}

\begin{proof}
For nonzero $v$, the map
\[
    \lambda\longmapsto\lambda\otimes v
\]
is injective by
Proposition~\ref{prop:multiplicative-cancellation-field}.

By Proposition~\ref{prop:scalar-inverse-generation},
\[
    I_{\oplus_V}(v)
    =
    \{\lambda\cdot v:
    \lambda\in I_{\oplus_F}({^{1}1})\}.
\]

By Proposition~\ref{prop:additive-inverses},
\[
    I_{\oplus_F}({^{1}1})
    =
    \{{^{j}1}:j=2,\ldots,s\}.
\]

Substitution gives the second equality.
\end{proof}

\begin{remark}
The construction in Theorem~\ref{s-vector_space}
defines a signed-vector space over its own
scalar field.

For $s=2$, it is a classical vector space.
For $s>2$, vector addition remains
non-associative and every nonzero vector
admits exactly $s-1$ additive inverses.
\end{remark}

\subsection[ni-Signed-Algebra]{\texorpdfstring{$n$i-Signed-Algebra}{ni-Signed-Algebra}}

\begin{definition}\label{def:ni-signed-algebra}
Let $F$ be an $n$i-signed-field.

An \textbf{$n$i-signed-algebra} over $F$ is an
$n$i-signed-vector space $(A,\oplus_A)$ over $F$,
equipped with a binary multiplication
\[
    \star:A\times A\longrightarrow A
\]
satisfying the following conditions for all
$a,b,c\in A$ and $\lambda\in F$:

\begin{enumerate}[label=\textnormal{(\roman*)}]

    \item Distributivity over addition:
    \begin{align*}
        a\star(b\oplus_A c)
        &=
        (a\star b)\oplus_A(a\star c),\\
        (a\oplus_A b)\star c
        &=
        (a\star c)\oplus_A(b\star c).
    \end{align*}

    \item Compatibility with scalar multiplication:
    \[
        \lambda\cdot(a\star b)
        =
        (\lambda\cdot a)\star b
        =
        a\star(\lambda\cdot b).
    \]

\end{enumerate}

We call $\star$ bilinear if it satisfies
both conditions.
\end{definition}

\begin{remark}
The operation $\star$ denotes multiplication
within $A$, whereas $\cdot$ denotes the
scalar action of $F$ on $A$.

An $n$i-signed-algebra is associative if
$\star$ is associative, commutative if
$\star$ is commutative, and unital if
$\star$ admits a bilateral identity.

The definition does not require these
properties.
\end{remark}

\begin{proposition}\label{prop:algebra-zero-absorption}
In every $n$i-signed-algebra $A$ over $F$,
the additive identity is absorbing under
$\star$.

For every $a\in A$,
\[
    0_A\star a=a\star0_A=0_A.
\]
\end{proposition}

\begin{proof}
By the zero scalar action and compatibility
with scalar multiplication,
\begin{align*}
    0_A\star a
    &=
    (0_F\cdot a)\star a\\
    &=
    0_F\cdot(a\star a)\\
    &=
    0_A.
\end{align*}

Similarly,
\begin{align*}
    a\star0_A
    &=
    a\star(0_F\cdot a)\\
    &=
    0_F\cdot(a\star a)\\
    &=
    0_A.
\end{align*}
\end{proof}

\begin{proposition}\label{prop:algebra-inverse-preservation}
Let $A$ be an $n$i-signed-algebra over $F$.
For all $a,b\in A$,
\[
    \{a\star c:c\in I_{\oplus_A}(b)\}
    \subseteq I_{\oplus_A}(a\star b)
\]
and
\[
    \{c\star a:c\in I_{\oplus_A}(b)\}
    \subseteq I_{\oplus_A}(b\star a).
\]
\end{proposition}

\begin{proof}
Let $c\in I_{\oplus_A}(b)$.
By distributivity and
Proposition~\ref{prop:algebra-zero-absorption},
\begin{align*}
    (a\star b)\oplus_A(a\star c)
    &=
    a\star(b\oplus_A c)\\
    &=
    a\star0_A\\
    &=
    0_A.
\end{align*}

Similarly,
\begin{align*}
    (b\star a)\oplus_A(c\star a)
    &=
    (b\oplus_A c)\star a\\
    &=
    0_A\star a\\
    &=
    0_A.
\end{align*}

By commutativity of addition, both
inclusions follow.
\end{proof}

\begin{theorem}\label{algebra-example}
The structure
\[
    A=(\sigmac^{s},\oplus,\otimes)
\]
is an associative, commutative, and unital
$(s-1)$i-signed-algebra over
\[
    F=(\sigmac^{s},\oplus,\otimes),
\]
with scalar action and internal multiplication
defined by
\[
    \lambda\cdot a=\lambda\otimes a,
    \qquad
    a\star b=a\otimes b.
\]
\end{theorem}

\begin{proof}
By Theorem~\ref{s-vector_space},
$(\sigmac^{s},\oplus)$ is an
$(s-1)$i-signed-vector space over $F$.

By Proposition~\ref{prop:multisign-distributivity},
\begin{align*}
    a\star(b\oplus c)
    &=
    (a\star b)\oplus(a\star c),\\
    (a\oplus b)\star c
    &=
    (a\star c)\oplus(b\star c).
\end{align*}

By associativity and commutativity of $\otimes$,
\begin{align*}
    \lambda\cdot(a\star b)
    &=
    \lambda\otimes(a\otimes b)\\
    &=
    (\lambda\otimes a)\otimes b\\
    &=
    (\lambda\cdot a)\star b.
\end{align*}

Similarly,
\begin{align*}
    \lambda\cdot(a\star b)
    &=
    \lambda\otimes(a\otimes b)\\
    &=
    a\otimes(\lambda\otimes b)\\
    &=
    a\star(\lambda\cdot b).
\end{align*}

Thus, $\star$ satisfies both conditions of
Definition~\ref{def:ni-signed-algebra}.

By Theorem~\ref{1-abelian_multiplication_proof},
$\star$ is associative and commutative,
with identity ${^{1}1}$.

Therefore, $A$ is an associative, commutative,
and unital $(s-1)$i-signed-algebra over $F$.
\end{proof}

\begin{remark}
The associativity of the signed-algebra in
Theorem~\ref{algebra-example} refers to its
internal multiplication $\star$.

For $s>2$, addition remains non-associative.
For $s=2$, both operations are associative,
and the construction is a classical algebra
over its scalar field.
\end{remark}

\section{Ordered Fields as Particular Cases of Multisign Algebra}
\label{real_numbers}

\begin{theorem}\label{thm:ordered-field-completion}
Let $P$ be a semiring with invertible multiplication
satisfying Assumption~\ref{ass:P-ordered}.
Then
\[
    K_P=(\sigmac_P^{2},\oplus,\otimes)
\]
is a classical field.

Moreover, $K_P$ admits a total order compatible
with its field operations, whose nonnegative cone is
\[
    K_{P,\geq 0}
    =
    \{{^{0}0}\}
    \cup
    \{{^{1}a}:a\in P^{\times}\}.
\]

The mapping
\[
    \iota:P\longrightarrow K_{P,\geq 0},
    \qquad
    \iota(a)=
    \begin{cases}
        {^{0}0}, & a=0,\\
        {^{1}a}, & a>0,
    \end{cases}
\]
is an isomorphism of ordered semirings.
\end{theorem}

\begin{proof}
By Theorem~\ref{s-field},
$(\sigmac_P^{2},\oplus,\otimes)$ is a
$1$i-signed-field.

By Remark~\ref{classical-s2-local},
addition is fully associative for $s=2$.
Therefore, $K_P$ is a classical field.

Define
\[
    C=\{{^{1}a}:a\in P^{\times}\}.
\]

By Definition~\ref{addition} and
Definition~\ref{multiplication},
for every $a,b>0$,
\[
    {^{1}a}\oplus{^{1}b}
    =
    {^{1}(a+b)}\in C
\]
and
\[
    {^{1}a}\otimes{^{1}b}
    =
    {^{1}(ab)}\in C.
\]

Thus, $C$ is closed under addition
and multiplication.

By Proposition~\ref{prop:additive-inverses},
the additive inverse of ${^{1}a}$ is
${^{2}a}$, and conversely.

Consequently,
\[
    K_P=C\sqcup\{{^{0}0}\}\sqcup(-C),
\]
where
\[
    -C=\{{^{2}a}:a\in P^{\times}\}.
\]

These properties establish that $C$ is a
positive cone of $K_P$.

Define the order on $K_P$ by
\[
    x\leq_{K_P}y
    \iff
    y-x\in C\cup\{{^{0}0}\},
\]
where we use subtraction in the classical field $K_P$.

The positive cone conditions ensure that this
is a total order compatible with addition
and multiplication in $K_P$.

By construction, $\iota$ is bijective.
Furthermore, for all $a,b\in P$,
\begin{align*}
    \iota(a+b)
    &=
    \iota(a)\oplus\iota(b),\\
    \iota(ab)
    &=
    \iota(a)\otimes\iota(b).
\end{align*}

The identities satisfy
\[
    \iota(0)={^{0}0},
    \qquad
    \iota(1)={^{1}1}.
\]

Finally, by the definition of addition and
the order on $K_P$,
\[
    a\leq b
    \iff
    \iota(a)\leq_{K_P}\iota(b).
\]

Therefore, $\iota$ is an isomorphism of
ordered semirings.
\end{proof}

\begin{theorem}\label{thm:ordered-field-isomorphism}
Let $K$ be an ordered field and let
\[
    P=K_{\geq 0}.
\]

Then the mapping
\[
    \varphi_K:\sigmac_P^{2}\longrightarrow K,
\]
defined by
\[
    \varphi_K({^{i}a})
    =
    \begin{cases}
        0, & i=0,\\
        a, & i=1,\\
        -a, & i=2,
    \end{cases}
\]
is an isomorphism of ordered fields, where
$\sigmac_P^{2}$ is equipped with the order
defined in Theorem~\ref{thm:ordered-field-completion}.
\end{theorem}

\begin{proof}
By Remark~\ref{example_semiring_with_invertible_multiplication},
the nonnegative cone $P$ satisfies the
magnitude assumptions.

Every element of $K$ is either zero,
positive, or negative.

For every $r\in K$, define
\[
    \psi_K(r)
    =
    \begin{cases}
        {^{0}0}, & r=0,\\
        {^{1}r}, & r>0,\\
        {^{2}(-r)}, & r<0.
    \end{cases}
\]

Since
\[
    \varphi_K\circ\psi_K=\operatorname{id}_K
\]
and
\[
    \psi_K\circ\varphi_K
    =
    \operatorname{id}_{\sigmac_P^{2}},
\]
the mapping $\varphi_K$ is bijective.

For operands with the same nonzero sign,
\begin{align*}
    \varphi_K({^{1}a}\oplus{^{1}b})
    &=a+b,\\
    \varphi_K({^{2}a}\oplus{^{2}b})
    &=-(a+b).
\end{align*}

For operands with different nonzero signs,
Definition~\ref{addition} gives
\[
    \varphi_K({^{1}a}\oplus{^{2}b})
    =
    \begin{cases}
        a-b, & a>b,\\
        0, & a=b,\\
        -(b-a), & a<b.
    \end{cases}
\]

In each case, the result equals $a-b$.
Together with the additive identity,
these cases establish
\[
    \varphi_K(x\oplus y)
    =
    \varphi_K(x)+\varphi_K(y)
\]
for all $x,y\in\sigmac_P^{2}$.

For nonzero operands, the multiplication
rule with $s=2$ gives
\begin{align*}
    {^{1}a}\otimes{^{1}b}&={^{1}(ab)},\\
    {^{1}a}\otimes{^{2}b}&={^{2}(ab)},\\
    {^{2}a}\otimes{^{1}b}&={^{2}(ab)},\\
    {^{2}a}\otimes{^{2}b}&={^{1}(ab)}.
\end{align*}

Applying $\varphi_K$ to these equalities
and considering the absorbing element
${^{0}0}$ gives
\[
    \varphi_K(x\otimes y)
    =
    \varphi_K(x)\varphi_K(y).
\]

Moreover,
\[
    \varphi_K({^{0}0})=0,
    \qquad
    \varphi_K({^{1}1})=1.
\]

Therefore, $\varphi_K$ is a field isomorphism.

Since $\varphi_K$ maps the positive cone
of $\sigmac_P^{2}$ onto the positive cone
of $K$, it also preserves the order.
\end{proof}

\begin{corollary}\label{cor:magnitude-characterization}
An ordered semiring satisfies
Definition~\ref{semiring_with_invertible_multiplication}
and Assumption~\ref{ass:P-ordered} if and only
if it is isomorphic, as an ordered semiring,
to the nonnegative cone of an ordered field.
\end{corollary}

\begin{proof}
The forward implication follows from
Theorem~\ref{thm:ordered-field-completion}.

Conversely, the nonnegative cone of every
ordered field satisfies the magnitude
conditions, as established in
Remark~\ref{example_semiring_with_invertible_multiplication}.

Ordered semiring isomorphisms preserve these conditions.
\end{proof}

\begin{corollary}[Isomorphism with the real numbers]
\label{isomorphism}
For $P=\mathbb{R}_{\geq 0}$,
\[
    (\sigmac_{\mathbb{R}_{\geq 0}}^{2},
    \oplus,\otimes)
    \cong
    (\mathbb{R},+,\cdot)
\]
as ordered fields.
\end{corollary}

\begin{proof}
Apply Theorem~\ref{thm:ordered-field-isomorphism}
with $K=\mathbb{R}$.
\end{proof}

\begin{corollary}\label{cor:other-ordered-fields}
The two-sign construction also satisfies
\[
    (\sigmac_{\mathbb{Q}_{\geq 0}}^{2},
    \oplus,\otimes)
    \cong
    (\mathbb{Q},+,\cdot)
\]
and
\[
    (\sigmac_{\mathbb{Q}(\sqrt{2})_{\geq 0}}^{2},
    \oplus,\otimes)
    \cong
    (\mathbb{Q}(\sqrt{2}),+,\cdot)
\]
as ordered fields.
\end{corollary}

\begin{proof}
Apply Theorem~\ref{thm:ordered-field-isomorphism}
to $\mathbb{Q}$ and $\mathbb{Q}(\sqrt{2})$,
respectively, with their orders inherited
from $\mathbb{R}$.
\end{proof}

\begin{corollary}\label{cor:classical-subfield-isomorphism}
Let $P$ satisfy Definition~\ref{semiring_with_invertible_multiplication}
and Assumption~\ref{ass:P-ordered}, and let $s$ be even.
The greatest classical subfield $C_s$
of $\sigmac_P^{s}$, established in
Theorem~\ref{thm:classical-subfield}, is
isomorphic to $\sigmac_P^{2}$.
\end{corollary}

\begin{proof}
Set $q=1+s/2$ and define
\[
    \theta:C_s\longrightarrow\sigmac_P^{2}
\]
by
\[
    \theta({^{0}0})={^{0}0},
    \qquad
    \theta({^{1}a})={^{1}a},
    \qquad
    \theta({^{q}a})={^{2}a},
    \qquad a\in P^{\times}.
\]

The mapping is bijective and preserves addition:
both constructions add magnitudes for equal signs,
subtract magnitudes for distinct signs, and share
the same zero and cancellation rules.

Under multiplication, the sign residues
$0$ and $s/2$ form a subgroup of
$\mathbb{Z}/s\mathbb{Z}$ isomorphic to
$\mathbb{Z}/2\mathbb{Z}$.

Since $\theta$ also preserves magnitudes,
it preserves multiplication.

Therefore, $\theta$ is a field isomorphism.
\end{proof}

\subsection{Structural differences between two and multiple signs}

Throughout this subsection, we assume the full magnitude
conditions of Definition~\ref{semiring_with_invertible_multiplication}
and Assumption~\ref{ass:P-ordered}.

For $s=2$, the addition $\oplus$ is associative,
every element has a unique additive inverse,
and additive equations have unique solutions.

For $s>2$, addition is not associative
and each nonzero element has exactly
$s-1$ additive inverses.
Nevertheless, each two-sign subset $S_{ij}$
remains an abelian group under addition.

The multiplicative structure retains
associativity, commutativity, and unique
inverses for all nonzero elements,
independently of the number of signs.

Table~\ref{tab:classical-multisign-comparison}
summarizes these properties.

\begin{table}[htbp]
\centering
\caption{Algebraic properties of the multisign
construction for $s=2$ and $s>2$.}
\label{tab:classical-multisign-comparison}
\small
\begin{tabular}{lcc}
\hline
\textbf{Property} & $\boldsymbol{s=2}$
                  & $\boldsymbol{s>2}$ \\
\hline
Additive closure             & Yes & Yes \\
Additive commutativity       & Yes & Yes \\
Additive identity            & Yes & Yes \\
Additive associativity       & Yes & No \\
Signed-associativity         & Yes & Yes \\
Nonzero additive inverse count & $1$ & $s-1$ \\
Global additive cancellation & Yes & No \\
All equations $x\oplus u=v$ solvable & Yes & Yes \\
Unique solutions for all $x\oplus u=v$ & Yes & No \\
Multiplicative closure       & Yes & Yes \\
Multiplicative associativity & Yes & Yes \\
Multiplicative commutativity & Yes & Yes \\
Multiplicative identity      & Yes & Yes \\
Nonzero multiplicative inverse count & $1$ & $1$ \\
Bilateral distributivity     & Yes & Yes \\
Zero absorption              & Yes & Yes \\
Absence of nonzero zero divisors & Yes & Yes \\
Classical field structure    & Yes & No \\
\hline
\end{tabular}
\end{table}

By Corollary~\ref{cor:magnitude-characterization},
the full magnitude assumptions characterize the
nonnegative cones of ordered fields.

The nonclassical properties of the multisign
construction therefore arise from the interaction
between sign labels under addition, rather
than from a relaxation of the algebraic
properties of magnitudes.

\section{Conclusions and future work}\label{conclusions}

In this work, we develop Multisign Algebra
by extending the additive construction on
metric spiders to abstract magnitude monoids
and introducing an internal multiplication
under additional algebraic assumptions.
The resulting numerical system separates
sign labels from magnitudes and admits
$s\geq2$ nonzero signs.

We establish that multisign addition is commutative
and satisfies signed-associativity. Every nonzero
element admits exactly $s-1$ bilateral additive
inverses. This number attains the upper bound
imposed by the axioms of an $n$i-signed-group on the
same canonically labelled multisign set with additive identity ${^{0}0}$.

For $s>2$, addition is not fully associative and
does not satisfy the global additive cancellation
law. Nevertheless, every pair of distinct nonzero
signs determines an additive group isomorphic to
the Grothendieck group of the magnitude monoid.
These groups constitute the
$\binom{s}{2}$ maximal associative additive
subsemigroups of $\sigmac^{s}$.

We characterize the solutions of additive equations
and establish that each equation of the form
$x\oplus u=v$ admits either one or $s-1$ solutions.
For $s>2$, this nonuniqueness extends to
other additive equations.

Multiplication is associative and commutative,
and every nonzero element has a unique
multiplicative inverse. The nonzero multiplicative
group satisfies
\[
    \bigl(
        \sigmac^{s}\setminus\{{^{0}0}\},
        \otimes
    \bigr)
    \cong
    \mathbb{Z}/s\mathbb{Z}\times P^{\times}.
\]

This decomposition characterizes multiplicative
powers, element orders, and the solutions of
power equations. In particular, the elements
of magnitude $1$ constitute the torsion subgroup
of the nonzero multiplicative group.

We establish that multiplication distributes
over addition on both sides and that multiplication
by any nonzero element induces an automorphism
of the additive structure. We also characterize
the additive inverses of a nonzero element $x$
through the multiplicative relation
\[
    I_{\oplus}(x)
    =
    \{
        x\otimes u:
        u\in I_{\oplus}({^{1}1})
    \}.
\]

These properties relate the additive inverse
structure to the multiplicative group.

We define rings, fields, modules, vector spaces,
and algebras in the $n$i-signed setting, with
prescribed additive inverse counts and restricted
associativity. We establish their defining
properties for the multisign construction and
derive consequences of the scalar-action and
bilinearity axioms without assuming full additive
associativity or additive cancellation.

We also establish that the full magnitude assumptions
characterize the nonnegative cones of ordered fields. For any ordered field $K$, the two-sign
construction over $K_{\geq0}$ is isomorphic to
$K$:
\[
    (\sigmac_{K_{\geq0}}^{2},\oplus,\otimes)
    \cong
    (K,+,\cdot).
\]

The real field is a particular case of this construction.

For even $s$, the multisign construction
contains a greatest classical subfield
isomorphic to its two-sign counterpart.
For odd $s$, it contains no classical
subfield sharing its multiplicative identity.

These results distinguish the properties
that depend on the number of signs from
those preserved independently of $s$.
The construction retains single-valued
operations, bilateral distributivity, and
a multiplicative group of nonzero elements,
while allowing multiple additive inverses
and restricted additive associativity.

Future work includes the following directions:

\begin{enumerate}

    \item Characterize automorphisms, substructures,
    and congruences of multisign algebra.

    \item Investigate alternative multiplicative
    structures on sign labels and determine
    the conditions under which distributivity
    and multiplicative invertibility are preserved.

    \item Develop multidimensional multisign
    structures and determine the algebraic
    properties preserved under Cartesian
    products.

    \item Investigate linear independence,
    generating sets, and dimension in
    $n$i-signed-vector spaces.

    \item Study metric and geometric structures
    induced by multisign operations.

    \item Develop matrix representations and
    investigate the corresponding algebraic
    properties.

\end{enumerate}

These directions require definitions and
results that account for the absence of
full additive associativity and global
additive cancellation when $s>2$.

\section*{Declarations}  
\subsection*{Conflicts of Interest}
The authors declare no conflicts of interest.
\subsection*{Authors' contributions} 
\textbf{Sebastián Aliaga-Rojas}: Conceptualization, Methodology, Formal analysis, Writing -- original draft, Writing -- review \& editing.\\ 
\textbf{Pamela Landero-Sepúlveda}: Methodology, Validation, Writing -- review \& editing.\\ 
\textbf{Mario Inostroza-Ponta}: Supervision, Validation, Writing -- review \& editing.
\subsection*{Funding}
The Agencia Nacional de Investigación y
Desarrollo (ANID), Chile, supported this work through the National Doctorate Scholarship
[grant number 2961-2022] and the FONDECYT Regular project
[grant number 1231505]. The Coordenação de Aperfeiçoamento de Pessoal de Nível Superior (CAPES),
Brazil, also partially funded this study [Finance Code 001].
\subsection*{Data Availability}
The authors generated and analyzed no datasets in this study.
Data sharing does not apply.

\section*{ORCID}
\noindent Sebastián Aliaga-Rojas - \url{https://orcid.org/0009-0003-1783-8555}

\noindent Pamela Landero-Sepúlveda - \url{https://orcid.org/0009-0005-4737-6384}

\noindent Mario Inostroza-Ponta - \url{https://orcid.org/0000-0003-1295-8972}

\section*{Acknowledgments}\label{sec:thanks}
We thank José Fernández Goycoolea and Mónica Villanueva Ilufi
for reviewing earlier versions of the manuscript and for their
suggestions, which helped improve the presentation.
We thank Professor Márcio Dorn and the Structural
Bioinformatics and Computational Biology Laboratory -- SBCB
of the Federal University of Rio Grande do Sul, Institute of
Informatics, for hosting us during our doctoral internship,
during which we developed part of this work.
The authors used ChatGPT by OpenAI to assist with writing. They reviewed and revised the content and take full responsibility for the manuscript.

\bibliographystyle{spmpsci}
\bibliography{bibliography}

@article{Anderson2019,
  author     = {Laura Anderson},
  title      = {Vectors of matroids over tracts},
  journal    = {Journal of Combinatorial Theory, Series A},
  year       = {2019},
  volume     = {161},
  pages      = {236--270},
  doi        = {10.1016/j.jcta.2018.08.002},
  url        = {https://www.sciencedirect.com/science/article/pii/S0097316518301109},
  issn       = {0097-3165}
}

@article{Ayupov1997,
  author     = {Sh. Ayupov and B. Omirov},
  title      = {On Leibniz algebras},
  journal    = {Dokl. Akad. Nauk Resp. Uzbekistan. Matematika},
  year       = {1997}
}

@article{Baker2019,
  author     = {Matthew Baker and Nathan Bowler},
  title      = {Matroids over partial hyperstructures},
  journal    = {Advances in Mathematics},
  year       = {2019},
  volume     = {343},
  pages      = {821--863},
  doi        = {10.1016/j.aim.2018.12.004},
  url        = {https://www.sciencedirect.com/science/article/pii/S0001870818304961},
  issn       = {0001-8708}
}

@article{Bowler2026,
  author     = {Nathan Bowler and Rudi Pendavingh},
  title      = {Perfect matroids over skew hyperfields},
  journal    = {European Journal of Combinatorics},
  year       = {2026},
  volume     = {134},
  pages      = {104327},
  doi        = {10.1016/j.ejc.2025.104327},
  url        = {https://www.sciencedirect.com/science/article/pii/S0195669825002203},
  issn       = {0195-6698}
}

@article{Casas2019,
  author     = {J. M. Casas and T. Datuashvili and M. Ladra},
  title      = {Action theory of alternative algebras},
  journal    = {Georgian Math. J.},
  volume     = {26},
  year       = {2019},
  pages      = {177--197}
}

@article{Chapoton2013,
  author     = {F. Chapoton and F. Patras},
  title      = {Enveloping algebras of pre-Lie algebras, Solomon idempotents and the Magnus formula},
  journal    = {Int. J. Algebra Comput.},
  volume     = {23},
  year       = {2013},
  pages      = {853--861}
}

@unpublished{vonEitzen2009,
  author     = {H. von Eitzen},
  title      = {Understanding polysign numbers the standard way},
  note       = {Unpublished manuscript},
  year       = {2009},
  url        = {http://www.von-eitzen.de/math/PolysignNumbers.pdf}
}

@article{Evans1970,
  author     = {Trevor Evans},
  title      = {Lattice-ordered loops and quasigroups},
  journal    = {Journal of Algebra},
  year       = {1970},
  volume     = {16},
  number     = {2},
  pages      = {218--226},
  doi        = {10.1016/0021-8693(70)90026-8},
  url        = {https://www.sciencedirect.com/science/article/pii/0021869370900268},
  issn       = {0021-8693}
}

@article{Figula2020,
  author     = {Figula, {\'A}gota and Nagy, P{\'e}ter T},
  title      = {Inverse property of nonassociative abelian extensions},
  journal    = {Commentationes Mathematicae Universitatis Carolinae},
  year       = {2020},
  volume     = {61},
  number     = {4},
  pages      = {501--511},
  publisher  = {Charles University in Prague, Faculty of Mathematics and Physics},
  url        = {http://eudml.org/doc/297165}
}

@misc{polysignUsenet,
  author     = {T. Golden and others},
  title      = {Discussions on polysign numbers},
  note       = {Informal communications on the sci.math newsgroup},
  year       = {2000s},
  url        = {https://groups.google.com/g/sci.math/search?q=polysign}
}

@article{Grabowski2022,
  author     = {Janusz Grabowski and Zohreh Ravanpak},
  title      = {Nonassociative analogs of Lie groupoids},
  journal    = {Differential Geometry and its Applications},
  year       = {2022},
  volume     = {82},
  pages      = {101887},
  doi        = {10.1016/j.difgeo.2022.101887},
  url        = {https://www.sciencedirect.com/science/article/pii/S0926224522000407},
  issn       = {0926-2245}
}

@book{Iachello2014,
  author     = {F. Iachello},
  title      = {Lie Algebras and Applications},
  series     = {Lecture Notes in Physics},
  publisher  = {Springer},
  year       = {2014}
}

@article{Im2021,
  author     = {Bokhee Im and Alex W. Nowak and Jonathan D.H. Smith},
  title      = {Algebraic properties of quantum quasigroups},
  journal    = {Journal of Pure and Applied Algebra},
  year       = {2021},
  volume     = {225},
  number     = {3},
  pages      = {106539},
  doi        = {10.1016/j.jpaa.2020.106539},
  url        = {https://www.sciencedirect.com/science/article/pii/S0022404920302401},
  issn       = {0022-4049}
}

@article{Im2024,
  author     = {Bokhee Im and Alex W. Nowak and Jonathan D.H. Smith},
  title      = {Symmetry classes of quantum quasigroups},
  journal    = {Journal of Pure and Applied Algebra},
  year       = {2024},
  volume     = {228},
  number     = {11},
  pages      = {107722},
  doi        = {10.1016/j.jpaa.2024.107722},
  url        = {https://www.sciencedirect.com/science/article/pii/S0022404924001191},
  issn       = {0022-4049}
}

@article{Jedlicka2017,
  author     = {Přemysl Jedlička and David Stanovský and Petr Vojtěchovský},
  title      = {Distributive and trimedial quasigroups of order 243},
  journal    = {Discrete Mathematics},
  year       = {2017},
  volume     = {340},
  number     = {3},
  pages      = {404--415},
  doi        = {10.1016/j.disc.2016.08.022},
  url        = {https://www.sciencedirect.com/science/article/pii/S0012365X16302801},
  issn       = {0012-365X}
}

@book{Jordan1933,
  author     = {Jordan, Pascual},
  title      = {{\"U}ber Verallgemeinerungsm{\"o}glichkeiten des Formalismus der Quantenmechanik},
  publisher  = {Weidmann},
  year       = {1933},
  address    = {Berlin}
}

@article{Jun2018,
  author     = {Jaiung Jun},
  title      = {Algebraic geometry over hyperrings},
  journal    = {Advances in Mathematics},
  year       = {2018},
  volume     = {323},
  pages      = {142--192},
  doi        = {10.1016/j.aim.2017.10.043},
  url        = {https://www.sciencedirect.com/science/article/pii/S0001870817303158},
  issn       = {0001-8708}
}

@article{Kunen1996,
  author     = {Kenneth Kunen},
  title      = {Quasigroups, Loops, and Associative Laws},
  journal    = {Journal of Algebra},
  year       = {1996},
  volume     = {185},
  number     = {1},
  pages      = {194--204},
  doi        = {10.1006/jabr.1996.0321},
  url        = {https://www.sciencedirect.com/science/article/pii/S002186939690321X},
  issn       = {0021-8693}
}

@article{Moufang1935,
  author     = {R. Moufang},
  title      = {Zur Struktur von Alternativk\"orpern},
  journal    = {Math. Ann.},
  volume     = {110},
  year       = {1935},
  pages      = {416--430}
}

@article{Phillips2010,
  author     = {J.D. Phillips and David Stanovský},
  title      = {Automated theorem proving in quasigroup and loop theory},
  journal    = {AI Communications},
  year       = {2010},
  volume     = {23},
  number     = {2-3},
  pages      = {267--283},
  doi        = {10.3233/aic-2010-0460},
  url        = {https://journals.sagepub.com/doi/abs/10.3233/AIC-2010-0460}
}

@article{Velrajan2010,
  author     = {Velrajan, Muthusamy and Asokkumar, Arjunan},
  title      = {Note on isomorphism theorems of hyperrings},
  journal    = {International Journal of Mathematics and Mathematical Sciences},
  year       = {2010},
  volume     = {2010},
  publisher  = {John Wiley \& Sons, Inc.}
}

@misc{Viro2010,
  author     = {Oleg Viro},
  title      = {Hyperfields for Tropical Geometry I. Hyperfields and dequantization},
  year       = {2010},
  url        = {https://arxiv.org/abs/1006.3034}
}

@article{Vishne2021,
  author  = {Vishne, Uzi},
  title   = {Four halves of the inverse property in loop extensions},
  journal = {Quasigroups and Related Systems},
  year    = {2021},
  volume  = {29},
  number  = {2},
  pages   = {283--302}
}

@article{Basrak2010,
  author = {Basrak, Bojan},
  title = {Limit Theorems for the Inductive Mean on Metric Trees},
  journal = {Journal of Applied Probability},
  year = {2010},
  volume = {47},
  number = {4},
  pages = {1136--1149}
}

@book{Weibel2013,
  author = {Weibel, Charles A.},
  title = {The {$K$}-book: An Introduction to Algebraic {$K$}-theory},
  series = {Graduate Studies in Mathematics},
  volume = {145},
  publisher = {American Mathematical Society},
  address = {Providence, RI},
  year = {2013},
  isbn = {978-0-8218-9132-2}
}

\end{document}